\documentclass{article}

\usepackage{amsmath,amsthm,amsfonts,amssymb,oldgerm}
\usepackage[frenchb,english]{babel}
\usepackage{indentfirst}
\numberwithin{equation}{section}
\usepackage[citecolor=blue,colorlinks=true]{hyperref}
\usepackage{geometry,color}
\allowdisplaybreaks[1]
\newtheorem{theorem}{Theorem}[section]
\newtheorem{proposition}[theorem]{Proposition}
\newtheorem{lemma}[theorem]{Lemma}
\newtheorem*{lemmanonnum}{Lemma}

\theoremstyle{remark}

\theoremstyle{definition}

\renewcommand\d{\partial}
\def\eps{\varepsilon }
\newcommand{\dd}{\textrm{d}}

\newcommand{\E}{\mathbb{E}}
\newcommand{\N}{\mathbb{N}}

\newcommand{\R}{\mathbb{R}}
\newcommand{\T}{\mathbb{T}}
\newcommand{\Z}{\mathbb{Z}}
\newcommand{\PP}{\mathbb{P}}

\newcommand\cF{{\mathcal F}}

\newcommand{\fe}{f^{\eps}}
\newcommand{\fd}{f^{\delta}}
\newcommand{\frd}{f^{R,\delta}}
\newcommand{\re}{r^{\eps}}
\newcommand{\ve}{g^{\eps}}

\newcommand{\feth}{f^{\eps}_3}
\newcommand{\li}[1]{\overline{#1}}
\newcommand{\ti}[1]{\widehat{#1}}

\begin{document}
\begin{center}
\textbf{DIFFUSION LIMIT FOR THE RADIATIVE TRANSFER EQUATION PERTURBED BY A WIENER PROCESS}
\end{center}

\vspace{0.2cm}
\begin{center}
\small \sc{A. Debussche\footnote[1]{\label{ref}IRMAR, ENS Rennes, CNRS, UEB. av Robert Schuman, F-35170 Bruz, France. Email: arnaud.debussche@ens-rennes.fr; sylvain.demoor@ens-rennes.fr}, S. De Moor\textsuperscript{\ref{ref}} and J. Vovelle\footnote[2]{Universit\'e de Lyon ; CNRS ; Universit\'e Lyon 1, Institut Camille Jordan, 43 boulevard du 11 novembre 1918, F-69622 Villeurbane Cedex, France. Email: vovelle@math.univ-lyon1.fr}}
\end{center}
\vspace{0.2cm}

\begin{abstract}
\footnotesize The aim of this paper is the rigorous derivation of a stochastic non-linear diffusion equation from a radiative transfer equation perturbed with a random noise. The proof of the convergence relies on a formal Hilbert expansion and the estimation of the remainder. The Hilbert expansion has to be done up to order 3 to overcome some difficulties caused by the random noise.\\

\end{abstract}

\normalsize

\section{Introduction}
\noindent In this paper, we are interested in the following non-linear equation
\begin{equation}\label{rt}
\begin{aligned}
& \dd \fe\ +\ \frac{1}{\eps}a(v)\cdot\nabla_x \fe\ \dd t\ =\ \frac{1}{\eps^2}\sigma(\li{\fe})L(\fe)\dd t\ +\ \fe\circ Q\dd W_t, \\ 
& \fe(0)=\rho_{\text{in}}, \qquad t\in[0,T],\, x\in \T^N,\,v\in V.
\end{aligned} 
\end{equation}
where $V$ is an $N$-dimensional torus, $a:V\to V$ and $\sigma:\R\to\R$. The notation $\li{f}$ stands for the average over the velocity space $V$ of the function $f$, that is 
$$\li{f}=\int_V f\,\dd v.$$
The operator $L$ is a linear operator of relaxation which acts on the velocity variable $v\in V$ only. It is given by
\begin{equation}\label{defL}
L(f):=\li{f}-f.
\end{equation}
The random noise term $W$ is a cylindrical Wiener process on the Hilbert space $L^2(\T^N)$. The covariance operator $Q$ is a linear self-adjoint operator on $L^2(\T^N)$. The precise description of the problem setting will be given in the next section. In this paper, we investigate the behaviour in the limit $\eps\to 0$ of the solution $\fe$ of $(\ref{rt})$. \medskip

\noindent Concerning the physical background in the deterministic case ($Q \equiv 0$), the equation \eqref{rt} describes the interaction between a surrounding continuous medium and a flux of photons radiating through it in the absence of hydrodynamical motion. The unknown $f^{\eps}(t,x,v)$ then stands for a distribution function of photons having position $x$ and velocity $v$ at time $t$. The function $\sigma$ is the opacity of the matter. When the surrounding medium becomes very large compared to the mean free paths $\eps$ of photons, the solution $\fe$ to \eqref{rt} is known to behave like $\rho $ where $\rho$ is the solution of the Rosseland equation 
$$
\partial_t \rho - \mathrm{div}_x(\sigma(\rho)^{-1}K\nabla_x\rho)\ =\ 0, \qquad (t,x)\in  [0,T] \times \T^N.
$$
with $K:=\int_Va(v)\otimes a(v)\,\dd v$. This is what is called the Rosseland approximation. In this paper, we investigate such an approximation where we have perturbed the deterministic equation by a smooth multiplicative random noise of the form $\fe\circ Q\dd W$. Note in particular that the noise is independant of the scaling $\eps$ of the equation. In the deterministic case, the Rosseland approximation has been widely studied. In the paper of Bardos, Golse and Perthame \cite{rosseland}, they derive the Rosseland approximation on a slightly more general equation of radiative transfer type than \eqref{rt} where the solution also depends on the frequency variable $\nu$. Using the so-called Hilbert's expansion method, they prove a strong convergence of the solution to the radiative transfer equation to the solution to the Rosseland equation. In \cite{bardos88}, the stationary and evolution Rosseland approximation are proved in a weaker sense with weakened hypothesis on the various parameters of the radiative transfer equation, in particular on the opacity function $\sigma$. \medskip

\noindent In the stochastic setting, the paper of Debussche and Vovelle \cite{arnaudjulien} deals with the problem of the radiative transfer equation where the opacity function is constant ($\sigma\equiv 1$) and with a multiplicative noise of the form $\frac{1}{\eps}\fe m^{\eps}$ where $m^{\eps}(t,x)=m(t/\eps^2,x)$ with $m$ a stationary Markov process. Note that in this setting, the noise also depends on the scaling $\eps$ of the equation and that formally $\frac{1}{\eps}m^{\eps}\dd t$ converges in law to some Wiener process $Q\dd W_t$ where $Q$ is a covariance operator which can be expressed in terms of the driving process $m$. In the paper \cite{arnaudjulien}, the authors prove the convergence in law of the solution to \eqref{rt} to a limit stochastic fluid equation by mean of a generalization of the perturbed test-functions method. \medskip

\noindent In this present work, we consider a non-linear operator $\sigma(\li{f})Lf$, which can be seen as a simple non-linear perturbation of the classical linear relaxation operator $L$ considered in \cite{arnaudjulien}. Nevertheless, we consider that the noise is already in its limit form $Q\dd W$. In particular, we point out that the fact that the noise is already in an It\^o form permits the application of the It\^o formula. As a consequence, we are able to prove in this paper a stronger result of convergence of $\fe$ to $\rho$, namely a strong convergence in the space $X:=L^{\infty}(0,T;L^1(\Omega;L^1_{x,v}))$ with rate $\eps$. The proof relies on the so-called Hilbert expansion method: we expand the solution $\fe$ to \eqref{rt} as $\fe = \rho + \eps f_1 + \eps^2 f_2 + \eps^3 f_3 + \re$ where $\rho$ is the solution to the limit problem, $f_1$, $f_2$, $f_3$ are three correctors to be defined appropriately and where $\re$ denotes the remainder of the expansion. First, we prove that the correctors $(f_i)_{1\leq i \leq 3}$ behave correctly in the space $X$ so that $\eps f_1 + \eps^2f_2 + \eps ^3 f_3 = O(\eps)$ in $X$. This step requires some regularity on the limit solution $\rho$ and we make use of the regularity result in \cite{regularity}. Then, to conclude the proof, we estimate the remainder by mean of an It\^o formula to show that $\re$ is of order $\eps$ in $X$. Note that an Hilbert expansion up to order $2$ is usually sufficient in many well-known deterministic cases; here we need to push the expansion up to order $3$ to overcome some difficulties caused by the noise term. \medskip

\noindent We point out that, in the sequel, when proving existence and uniqueness for the problem \eqref{rt}, we use a stochastic averaging lemma which can be interesting by itself. It provides a better regularity for the average over the velocity space of solutions to kinetic stochastic equations, see Lemma \ref{lemmemoyenne}. The proof of this lemma is detailed in Appendix B; it is mainly based on an adaptation to a stochastic setting of the paper of Bouchut and Desvillettes \cite{bouchut}. \medskip

\noindent The paper is organized as follows. In Section \ref{sec1}, we introduce the setting and the notations and give the main result to be proved, Theorem \ref{mainresult}. In Section \ref{sec2}, we derive formally the limit equation. Finally, in Section \ref{sec3}, we provide the proof of the main result, which is divided in three main steps. First, we study the existence, uniqueness and regularity of the solutions to the radiative transfer equation \eqref{rt} and to the stochastic Rosseland problem. Then we define and study the correctors of the Hilbert expansion. Finally, we estimate the remainder to conclude the proof. \medskip

\noindent {\em Aknowledgements:} This work is partially supported by the french government thanks to the ANR program Stosymap. It also benefit from the support of the french government ``Investissements d'Avenir'' program ANR-11-LABX-0020-01.

\section{Preliminaries and main result}\label{sec1}
\subsection{Notations and hypothesis}\label{sec:nothyp}

\noindent Let us now introduce the precise setting of equation \eqref{rt}. We work on a finite-time interval $[0,T],\,T>0,$ and consider periodic boundary conditions for the space variable: $x\in\T^N$ where $\T^N$ is the $N$-dimensional torus. Regarding the velocity space $V$, we also consider periodic boundary conditions, that is $V=\T^N$, but we keep the notation $V$ to distinguish the velocity space from the space one. \medskip

\noindent For $p\in[1,\infty]$, the Lebesgue spaces $L^p(\T^N\times V)$ will be denoted by $L^p_{x,v}$ for short. The associated norm will be written $\|\cdot\|_{L^p_{x,v}}$.  Similarly, we define the Lebesgue spaces $L^p_x$, $L^p_v$ and, if $k\in\Z$, the Sobolev spaces $W^{k,p}_{x,v}$ and $W^{k,p}_x$ or $H^k_{x,v}$ and $H^k_x$ when $p=2$. The scalar product of $L^2_{x,v}$ will be denoted by $(\cdot,\cdot)$. We finally introduce, for $k\in\N$, the space $C^{0,k}([0,T]\times\T^N)$ constituted by the functions of the variables $(t,x)\in[0,T]\times\T^N$ which are continuous in time and $k$-times continuously differentiable in space. \medskip

\noindent Concerning the velocity mapping $a:V\to V$, we shall assume that it is $C^1_b$. Furthermore, we suppose that the following null flux hypothesis holds
\begin{equation}\label{nullflux}
\int_Va(v)\,\dd v = 0.
\end{equation}
We also define the following matrix 
\begin{equation}\label{defK}
K:=a(v)\otimes a(v)
\end{equation}
and assume that $\li{K}$ is definite positive. Furthermore, we use a stochastic version of averaging lemmas to prove the existence of the solution $\fe$ to \eqref{rt}. To do so, we need to assume the following standard condition:
\begin{equation}\label{nondegenlemmemoy}
\forall \eps >0,\, \forall (\xi,\sigma)\in  S^{N-1} \!\!\times \R,\;\; \text{Leb}\left(\{v\in V, |a(v)\cdot\xi + \sigma|<\eps\}\right)\leq \eps^{\alpha},
\end{equation}
for some $\alpha \in (0,1]$ and where Leb denotes the normalized Lebesgue measure on $V=\T^N$. \medskip

\noindent Regarding the opacity function $\sigma:\R\to\R$, we assume that
\begin{enumerate}
\item[(H1)] There exist two positive constants $\sigma_*$, $\sigma^* >0$ such that for almost all $x\in\R$, we have 
$$\sigma_*\leq \sigma(x)\leq \sigma^*;$$

\item[(H2)] the function $\sigma$ is $\mathcal{C}^3_b$, in particular $\sigma$ is Lipschitz continuous; 

\item[(H3)] the mappings $x\mapsto\sigma(x)$ and $x\mapsto\sigma(x)x$ are respectively non-increasing and non-decreasing.
\end{enumerate}

\noindent Finally, the initial condition $\rho_{\text{in}}$ is supposed to be a smooth non-negative function which does not depend on the variable $v\in V$.

\subsection{The random noise}\label{sec:noise}
\noindent Regarding the stochastic term, let $(\Omega,\cF,(\cF_t)_{t\geq0},\PP)$ be a stochastic basis with a complete, right-continuous filtration. The random noise $\dd W_t$ is a cylindrical Wiener process on the Hilbert space $L^2(\T^N)$. We can define it by setting 
\begin{equation}
\dd W_t\ =\ \sum_{k\geq 0} e_k\ \dd\beta_k(t),
\end{equation}
where the $(\beta_k)_{k\geq 0}$ are independent Brownian motions on the real line and $(e_k)_{k\geq 0}$ a complete orthonormal system in the Hilbert space $L^2(\T^N)$. The covariance operator $Q$ is a linear self-adjoint operator on $L^2(\T^N)$. We assume the following regularity property 
\begin{equation}\label{regularnoise}
\sum_{k\geq 0}\|Qe_k\|^2_{W^{4,\infty}_x}<\infty.
\end{equation}
In particular, we define
\begin{equation}\label{regularnoise01}
\kappa_{0,\infty}:=\sum_{k\geq 0}\|Qe_k\|^2_{L^{\infty}_x}<\infty,\qquad \kappa_{1,\infty}:=\sum_{k\geq 0, \,1\leq i\leq N}\|\partial_{x_i}Qe_k\|^2_{L^{\infty}_x}<\infty.
\end{equation}
As a consequence, we can introduce
$$G:=\frac{1}{2}\sum_{k\geq 0}(Qe_k)^2,$$
which will be useful when switching Stratonovich integrals into It\^o form. Precisely, we point out that for Equation \eqref{rt} we can write $\fe\circ Q\dd W_t\ =\ \fe Q\dd W_t + G\fe\dd t$ where 
$$
Q\dd W_t\ =\ \sum_{k\geq 0} Qe_k\ \dd\beta_k(t).
$$
In the sequel, we will have to consider stochastic integrals of the form $h Q\dd W_t$ where $h\in L^p_{x,v}$, $p\geq 2$, and we should ensure the existence of the stochastic integrals as $L^p_{x,v}$-valued processes. We recall that the Lebesgue spaces $L^p_{x,v}$ with $p\geq 2$ belong to a class of the so-called $2$-smooth Banach spaces, which are well suited for stochastic It\^o integration (see \cite{gamma1}, \cite{gamma2} for a precise construction). So, let us denote by $\gamma(L^2(\T^N),X)$ the space of the $\gamma$-radonifying operators from $L^2(\T^N)$ to a $2$-smooth Banach space $X$. We recall that $\Psi\in\gamma(L^2(\T^N),X)$ if the series 
$$\sum_{k\geq 0}\gamma_k\Psi(e_k)$$
converges in $L^2(\widetilde{\Omega},X)$, for any sequence $(\gamma_k)_{k\geq 0}$ of independent normal real valued random variables on a probability space $(\widetilde{\Omega},\widetilde{\mathcal{F}},\widetilde{\PP})$. Then, the space $\gamma(L^2(\T^N),X)$ is endowed with the norm 
$$\|\Psi\|_{\gamma(K,X)}:=\Bigg{(}\widetilde{\E}\Bigg{|}\sum_{k\geq 0}\gamma_k\Psi(e_k)\Bigg{|}_X^2\Bigg{)}^{1/2}$$ 
(which does not depend on $(\gamma_k)_{k\geq 0}$) and is a Banach space. 
Now, if $h\in L^p_{x,v}$, $p\geq 2$, $h Q\dd W$ can be interpreted as $\Psi\dd W$ where $\Psi$ is the following $\gamma$-radonifying operator from $L^2(\T^N)$ to $L^p_{x,v}$:
$$\Psi(e_k) := h Q e_k.$$
Let us compute the $\gamma$-radonifying norm of $\Psi$. We fix $(\gamma_j)_{j\in\N}$ a sequence of independent $\mathcal{N}(0,1)$-random variables.
$$
\begin{aligned}
\|\Psi\|^2_{\gamma(L^2(\T^N),L^p_{x,v})}&\ =\ \widetilde{\E} \Big{\|}\sum_k\gamma_kh(e_k)\Big{\|}^2_{L^p_{x,v}}=\ \widetilde{\E} \Big{\|}\sum_k\gamma_kh Qe_k\Big{\|}^2_{L^p_{x,v}}\\
&\leq\ \Bigg{(}\widetilde{\E}\Big{\|}\sum_k\gamma_kh Qe_k\Big{\|}^p_{L^p_{x,v}}\Bigg{)}^{2/p}=\ \Bigg{(}\widetilde{\E}\int_{\T^N\times V}\Big{|}\sum_k\gamma_kh Qe_k\Big{|}^p\Bigg{)}^{2/p}.
\end{aligned}
$$
Observe that, almost everywhere in $\T^N\times V$, $\sum_k\gamma_kh Qe_k$ is a real centered Gaussian with covariance $\sum_k|h\, Qe_k|^2$. As a consequence, there exists a constant $C_p\in(0,\infty)$ such that
$$ \widetilde{\E}\Big{|}\sum_k\gamma_kh Qe_k\Big{|}^p=C_p\Big{(}\sum_k|h Qe_k|^2\Big{)}^{p/2}.$$
We use this equality in the computations of the $\gamma$-radonifying norm to obtain, thanks to $(\ref{regularnoise01})$,
\begin{equation}\label{gammanorm}
\begin{aligned}
\|\Psi\|^2_{\gamma(L^2(\T^N),L^p_{x,v})}&\leq\ C_p^{2/p}\Bigg{(}\int_{\T^N\times V}\Big{(}\sum_k(Qe_k)^2\Big{)}^{p/2}|h|^p\Bigg{)}^{2/p}\\
&\leq\ C_p^{2/p}\kappa_{0,\infty}\|h\|^2_{L^p_{x,v}}.
\end{aligned}
\end{equation}

\subsection{Properties of the operator $\sigma(\li{\cdot}) L(\cdot)$.}\label{sec:sigmaL}

\noindent Similarly as in the deterministic case, we expect with $(\ref{rt})$ that $\sigma(\li{\fe}) L(\fe)$ tends to zero with $\eps$, so that we should determine the equilibrium of the operator $\sigma(\li{\cdot}) L(\cdot)$. In this case, since $\sigma>0$, they are clearly constituted by the functions independent of $v\in V$. \\
In the space $L^2_{x,v}$, the operator $\sigma(\li{\cdot}) L(\cdot)$ is dissipative. Namely, we have, for $f\in L^2_{x,v}$,
\begin{equation}\label{sigmaLL2}
(\sigma(\li{f})Lf,f)=-\|\sigma(\li{f})^{1/2}Lf\|^2_{L^2_{x,v}}\leq 0.
\end{equation}
In the space $L^1_v$ we have some accretivity properties for the operator $\sigma(\li{\cdot}) L(\cdot)$. Namely, (see \cite{rosseland}), if $f$, $g\in L^1_v$ with $f\geq 0$, we have 
\begin{equation}\label{accrdet}
\int_{V}\text{sgn}^+(f-g)\left[\sigma(\li{f})L(f)-\sigma(\li{g})L(g)\right]\dd v \leq 0,
\end{equation}
where $\text{sgn}^+(x):=\mathbf{1}_{x\geq 0}$. In the deterministic setting, the quantity above is involved when deriving the equation satisfied by $(f-g)^+$ where $f$ and $g$ are solutions to the equation $(\ref{rt})$ without noise and where $x^+:=\max(0,x)$ stands for the positive part of $x$. This is the main argument that permits to prove uniqueness for equation $(\ref{rt})$ without noise. In our stochastic setting, this procedure will be replaced by the application of It\^o formula with the function $x\mapsto x^+$ to the process $f-g$. To make this plainly rigorous, we have to approximate the map $x\mapsto x^+$ by regular (at least $C^2$) functions. Therefore, we have to investigate what we have lost in the bound $(\ref{accrdet})$ above when replacing $\text{sgn}^+$ by some smooth approximation. To this end, take $\psi$ a smooth (at least $C^2$) non-decreasing function such that
$$
\left\{
\begin{array}{l}
\psi(x) = 0, \quad x\in (-\infty,0], \\
\psi(x) = 1, \quad x\in [1,+\infty), \\
0<\psi(x)<1, \quad x\in (0,1).
\end{array}
\right.
$$
and define 
\begin{equation}\label{defvarphi}
\varphi_{\delta}(x):=\int_0^x\psi\left(\frac{y}{\delta}\right)\,\dd y,\quad x\in\R.
\end{equation}
Then, we have the following lemma.
\begin{lemma}\label{accr}
Let $\delta>0$. Suppose that $f,\,g\in L^1_v$ with $f\geq 0$. We have the two following estimates
\begin{equation}\label{accr1}
\int_{V}\varphi_{\delta}'(f-g)\left[\sigma(\li{f})L(f)-\sigma(\li{g})L(g)\right]\dd v \leq C\left(1+\|f\|_{L^1_v} \right)\delta,
\end{equation}
\begin{equation}\label{accr2}
\int_{V}\varphi_{\delta}'(g-f)\left[\sigma(\li{g})L(g)-\sigma(\li{f})L(f)\right]\dd v \leq C\left(1+\|f\|_{L^1_v} \right)\delta.
\end{equation}
\end{lemma}
\begin{proof}
The proof is given in Appendix A.
\end{proof}

\subsection{Main result}

\noindent We may now state our main result, the proof of which will be given throughout this paper.
\begin{theorem}\label{mainresult}
Let $\fe$ denote the solution of the kinetic problem $(\ref{rt})$ in the sense of Proposition \ref{solkinetic} and $\rho$ the solution of the non-linear stochastic partial differential equation
\begin{equation}\label{EqF}
\left\{
\begin{aligned}
& \dd \rho - \mathrm{div}_x\left(\sigma(\rho)^{-1}\li{K}\,\nabla_x \rho \right) \dd t = \rho\circ Q\dd W_t,\\
& \rho(0)=\rho_{\text{in}},
\end{aligned}
\right.
\end{equation}
in the sense of Proposition \ref{existencerho} and where $K$ denotes the matrix \eqref{defK}. Then, the solution $\fe$ converges as $\eps$ tends to $0$ to the fluid limit $\rho$ and we have the estimate
\begin{equation}\label{mainestimate}
\sup\limits_{t\in[0,T]}\E\|\fe_t-\rho_t\|_{L^1_{x,v}}\leq C\eps.
\end{equation}
\end{theorem}

\section{Formal Hilbert expansion}\label{sec:formalhilbert}\label{sec2}

\noindent In this section, we derive formally the limit equation satisfied by $\fe$ as $\eps$ goes to $0$. To do so, we classically introduce the following Hilbert expansion of the solution $\fe$:
$$\fe=f_0+\eps f_1+\eps^2f_2 + ...$$
Then, discarding the terms with positive power of $\eps$, equation $(\ref{rt})$ reads 
$$
\begin{aligned}
\dd f_0 & =\  -\frac{1}{\eps}a(v)\cdot\nabla_x f_0\ \dd t -a(v)\cdot\nabla_x f_1\ \dd t
+ \frac{1}{\eps^2}\sigma(\li{f_0}+\eps \li{f_1}+\eps^2\li{f_2})L(f_0+\eps f_1+\eps^2f_2)\ \dd t \\
 & +\ f_0\circ Q\dd W_t + O(\eps).
\end{aligned}
$$
Putting the terms with the same power of $\eps$ together and omitting once again those with positive power of $\eps$, we have
$$
\begin{aligned}
\dd f_0 & =\ \frac{1}{\eps^2}\sigma(\li{f_0})L(f_0)\ \dd t + \left(-\frac{1}{\eps}a(v)\cdot\nabla_x f_0 +\frac{1}{\eps}\sigma(\li{f_0})L(f_1)\right. \\
&\qquad +\left.\frac{1}{\eps^2}\left[\sigma(\li{f_0}+\eps\li{f_1})-\sigma(\li{f_0})\right]L(f_0)\right) \dd t \\
& +\ \left( -a(v)\cdot\nabla_x f_1 + \sigma(\li{f_0})L(f_2) +\frac{1}{\eps^2}\left[\sigma(\li{f_0}+\eps\li{f_1}+\eps^2\li{f_2})-\sigma(\li{f_0}+\eps\li{f_1})\right]L(f_0)\right.  \\
& \qquad + \left.\frac{1}{\eps}\left[\sigma(\li{f_0}+\eps\li{f_1})-\sigma(\li{f_0})\right]L(f_1)\right)\dd t + f_0\circ Q\dd W_t+ O(\eps).
\end{aligned}
$$
Next, we identify the terms having the same power of $\eps$. At the order $\eps^{-2}$, we find $\sigma(\li{f_0})L(f_0) = 0$,  which implies $L(f_0)=0$; thus we have $f_0=\li{f_0}=:\rho$. Then, at the order $\eps^{-1}$, with the fact that $L(f_0)=0$, we find 
$$L(f_1) = \sigma(\rho)^{-1}\,a(v)\cdot\nabla_x \rho.$$
Since the integral with respect to $v\in V$ of the right-hand side vanishes thanks to $(\ref{nullflux})$, this equation can be solved by 
\begin{equation}\label{formalf1}
f_1:=-\sigma(\rho)^{-1}\,a(v)\cdot\nabla_x \rho,
\end{equation}
and we point out that $\li{f_1}=0$. Finally, at the order $\eps^0$, we get
\begin{equation}\label{formal3}
\dd \rho = -a(v)\cdot\nabla_x f_1\ \dd t + \sigma(\rho)L(f_2)\ \dd t + \rho\circ Q\dd W_t.
\end{equation}
By integration with respect to $v\in V$ and with $\int_VL(f_2)\dd v=0$, we discover 
$$\dd \rho = -\mathrm{div}_x(\li{a(v) f_1})\ \dd t + \rho\circ\dd W_t,$$
that is, thanks to the expression of $f_1$ given by $(\ref{formalf1})$,
\begin{equation}\label{formalEqF}
\dd \rho - \mathrm{div}_x\left(\sigma(\rho)^{-1}\li{K}\,\nabla_x \rho \right) \dd t = \rho\circ Q\dd W_t,
\end{equation}
where $K=a(v)\otimes a(v).$
Furthermore, if $\rho$ satisfies equation $(\ref{formalEqF})$, equation $(\ref{formal3})$ now reads
$$
\sigma(\rho)L(f_2) = \mathrm{div}_x\left(\sigma(\rho)^{-1}(\li{K}-K)\nabla_x \rho\right),
$$
and since the integral with respect to $v\in V$ of the right-hand side vanishes, this can indeed be solved by setting
\begin{equation}\label{formalf2}
f_2 := -\sigma(\rho)^{-1}\mathrm{div}_x\left(\sigma(\rho)^{-1}(\li{K}-K)\nabla_x \rho  \right).
\end{equation}
To conclude, the solution $\fe$ of the kinetic problem $(\ref{rt})$ formally converges to an equilibrium state $\rho$ which satisfies the \textit{non-linear} stochastic partial differential equation $(\ref{formalEqF})$ given above.

\section{Convergence of $\fe$}\label{sec3}

\noindent In this section, we now give a rigorous proof of the convergence of $\fe$. The main difficulty is that the remainder $\re:=\fe-\rho-\eps f_1-\eps^2f_2$ can only be appropriately estimated in $L^1_{x,v}$. As a result, in our stochastic case, we will need to apply It\^o formula in $L^1_{x,v}$. This gives rise to some difficulties. So, in the sequel, we will need to push the Hilbert expansion of $\fe$ up to order $3$ to overcome these problems. To begin with, we solve the kinetic problem $(\ref{rt})$ and the limiting equation $(\ref{EqF})$ and investigate the regularity and properties of the solutions.

\subsection{Resolution of the kinetic problem}

\noindent Let us study the kinetic problem $(\ref{rt})$. We solve it using a standard semigroup approach combined with a regularization of the random noise term. Let $p\in[1,\infty]$. We introduce the contraction semigroup $(\mathcal{U}(t))_{t\geq 0}$ generated by the linear operator $-a(v)\cdot\nabla_x$ on the space $L^p_{x,v}$.

\begin{proposition}\label{solkinetic}Let $\rho_{\text{in}}$ be a smooth non-negative function which does not depend on $v\in V$. Then there exists a unique non-negative strong It\^o solution $\fe$ to the kinetic problem $(\ref{rt})$ which belongs to  $L^2(\Omega;L^2(0,T;L^2_{x,v}))$ with $\nabla_x\fe\in L^2(0,t;L^2_{x,v})$ a.s. for all $t<T$, that is, $\PP-$a.s. for all $t\in[0,T]$,
$$\fe(t)=\rho_{\text{in}}-\frac{1}{\eps}\int_0^{t}a(v)\cdot\nabla_x\fe_s\, \dd s+\frac{1}{\eps}\int_0^{t}\sigma(\li{\fe_s})L(\fe_s)\, \dd s + \int_0^{t} G\fe_s\, \dd t + \int_0^{t}\fe_s\, Q\dd W_s.
$$
 Furthermore, we have the following uniform bound
\begin{equation}\label{unifbound}
\sup\limits_{t\in[0,T]}\E\|\fe(t)\|^2_{L^2_{x,v}}\leq C.
\end{equation}
\end{proposition}
\noindent Before giving the proof of the proposition, we recall a classical result about the regularization of the stochastic convolution.
\begin{lemma}\label{RegulStochastic}
Let $p\in(2,\infty)$. Let $\Psi\in L^p(\Omega;L^p(0,T;L^p_{x,v}))$. We define
$$
z(t):=\int_0^t\mathcal{U}(t-s)\Psi(s)\,Q\dd W_s,\quad t\in[0,T].
$$  
Then $z\in L^p(\Omega;C([0,T];L^p_{x,v}))$ and
$$\E\sup\limits_{t\in[0,T]}\|z(t)\|^p_{L^p_{x,v}}\leq C\,\E\|\Psi\|^p_{L^p(0,T;L^p_{x,v})},$$
for some constant $C$ which depends on $p$ and $\kappa_{0,\infty}$.
\end{lemma}
\noindent The proof relies on the so-called factorization method (see \cite[Section 11]{peszat}) combined with the application of the Burkholder-David-Gundy inequality for martingales with values in a $2$-smooth Banach space (see \cite{gamma1} and \cite{gamma2}) and the bound $(\ref{gammanorm})$.
\begin{proof} {\em Existence and uniqueness part.} In this part of the proof, for the sake of convenience, we set $\eps=1$.

{\em Step 1: Uniqueness.} We first begin with the proof of uniqueness for equation $(\ref{rt})$. So let $f$ and $g$ be two non-negative solutions of $(\ref{rt})$ with the same initial condition $\rho_{\text{in}}$ and which at least belong to $L^1(\Omega;L^1(0,T;L^1_{x,v}))$. We set $r:=f-g$ and estimate $r$ in $L^1_{x,v}$ by applying the It\^o formula with the $C^2$ function $\varphi_{\delta}$ defined by $(\ref{defvarphi})$ which approximates $x\mapsto x^+$. This gives (note that the term relative to $a(v)\cdot\nabla_x\re$ cancels)
$$
\begin{aligned}
\E\int_{\T^N\times V}\varphi_{\delta}(r_t)\ & =\ \E\int_0^t\int_{\T^N\times V}\varphi_{\delta}'(f_s-g_s)\left[\sigma(\li{f_s})L(f_s)-\sigma(\li{g_s})L(g_s)\right]\, \dd s \\
& +\ \E\int_0^t\int_{\T^N\times V}\varphi_{\delta}'(r_s)Gr_s\,\dd s +\ \E\int_0^t\int_{\T^N\times V}\varphi_{\delta}''(r_s)G|r_s|^2\,\dd s.
\end{aligned}
$$
Since $x^+\leq\varphi_{\delta}(x)+\delta$, we have
$$\E\|(r_t)^+\|_{L^1_{x,v}}\leq\ \E\int_{\T^N\times V}\varphi_{\delta}(r_t) + \delta.$$
Then, for the next term, we use the accretivity property of the operator $\sigma(\li{\cdot}) L(\cdot)$. Namely, with Lemma $\ref{accr}$, we get 
$$
\E\int_0^t\int_{\T^N\times V}\varphi_{\delta}'(f_s-g_s)\left[\sigma(\li{f_s})L(f_s)-\sigma(\li{g_s})L(g_s)\right] \leq\ C\delta\left(1+\E\int_0^T\|f_s\|_{L^1_{x,v}}\dd s\right)\leq C\delta.$$
For the following term, we just observe that $|\varphi_{\delta}'|\leq 1$ and that $\|G\|_{L^{\infty}_x}<\infty$ with $(\ref{regularnoise01})$ so that
$$
\E\int_0^t\int_{\T^N\times V}\varphi_{\delta}'(r_s)Gr_s\,\dd s\ \leq\ C\,\E\int_0^t\|r_s\|_{L^1_{x,v}}\,\dd s.
$$
For the last term of the It\^o formula, we point out that $\varphi_{\delta}''$ is zero on $[0,\delta]^c$ and that $|\varphi_{\delta}''|\leq 1/\delta$ on $[0,\delta]$. Thus, we obtain
$$
\E\int_0^t\int_{\T^N\times V}\varphi_{\delta}''(r_s)G|r_s|^2\,\dd s\ \leq\ C\delta.
$$
Summing up all the previous bounds now yields
$$\E\|(r_t)^+\|_{L^1_{x,v}} \leq\  C\delta+C\,\E\int_0^t\|r_s\|_{L^1_{x,v}}\,\dd s.$$
A similar work can be done for $(r)^-=(-r)^+$. As a result we obtain the estimate
$$\E\|r_t\|_{L^1_{x,v}} \leq\  C\delta+C\,\E\int_0^t\|r_s\|_{L^1_{x,v}}\,\dd s.$$
Since this inequality holds true for all $\delta>0$, an application of the Gronwall lemma yields $f=g$ in $L^1(\Omega;L^1(0,T;L^1_{x,v}))$.

{\em Step 2: Resolution of a regularized equation.} For $\delta>0$, we will denote by $\xi_{\delta}$ a mollifier on $\T^N\times V$ as $\delta\to 0$. This step is devoted to the proof of existence of a solution $\fd$ to the regularized equation
\begin{equation}\label{rtregul}
\begin{array}{cc}
\displaystyle \dd f\ +\ a(v)\cdot\nabla_x f\ \dd t\ =\ \sigma(\li{f})L(f)\ \dd t\ +\ Gf\ \dd t\ +\ f*\xi_{\delta}\ Q\dd W_t,
\end{array} 
\end{equation}
with $\delta>0$ being fixed. Let us fix $p>N$. We will apply a fixed point argument in the space $L^p(\Omega;C([0,T_0];L^{\infty}_{x,v}))$ with $T_0$ sufficiently small. Before doing this, we first need to truncate the equation to overcome with the non-linear term $f\mapsto\sigma(\li{f})Lf$ which is not Lipschitz. Following for example \cite{debdeb} or \cite{Gyongy}, we introduce $\theta\in C^{\infty}_0(\R)$ whose compact support is embedded in $(-2,2)$ and such that $\theta(x)=1$ for $x\in[-1,1]$ and $0\leq\theta\leq 1$ on $\R$. Then, for $R>0$, we set $\theta_R(x)=\theta(x/R)$. We are now considering the following equation:
\begin{equation}\label{rtregultrunc}
\begin{array}{cc}
\displaystyle \dd f\ +\ a(v)\cdot\nabla_x f\ \dd t\ =\ \theta_R(\|f\|_{L^{\infty}_{x,v}})\sigma(\li{f})L(f)\ \dd t\ +\ Gf\ \dd t\ +\ f*\xi_{\delta}\ Q\dd W_t,
\end{array} 
\end{equation}
and we are looking for a mild solution $\frd$, that is,
\begin{equation}\label{rtregultruncmild}
\begin{aligned}
f(t)\ =\ & \mathcal{U}(t)\rho_{\text{in}}+\int_0^t\mathcal{U}(t-s)\theta_R(\|f_s\|_{L^{\infty}_{x,v}})\sigma(\li{f_s})L(f_s)\, \dd s + \int_0^t \mathcal{U}(t-s)Gf_s\, \dd t \\
&\qquad +\ \int_0^t \mathcal{U}(t-s)f_s*\xi_{\delta}\, Q\dd W_s.
\end{aligned}
\end{equation}
Here, as usual, if $f\in L^p(\Omega;C([0,T_0];L^{\infty}_{x,v}))$, we denote by $\mathcal{T}f$ the right-hand side of the previous equation and we shall verify that the Banach fixed-point Theorem applies. We refer the reader to \cite[Proof of Proposition 3.1]{debdeb} for a precise proof in a similar setting. Here, we just prove the contraction property of the stochastic integral. Thanks to Lemma \ref{RegulStochastic} and with Young's inequality, we easily obtain
$$
\E\sup\limits_{t\in[0,T_0]}\left\|\int_0^t \mathcal{U}(t-s)(f_s-g_s)*\xi_{\delta}\, Q\dd W_s\right\|^p_{L^p_{x,v}} \leq C\, T_0\, \E\sup\limits_{s\in[0,T_0]}\|f_s-g_s\|^p_{L^{\infty}_{x,v}},
$$
where the constant $C$ depends on $p$ and $\kappa_{0,\infty}$. Now, since $\nabla_x\mathcal{U}(t)g=\mathcal{U}(t)\nabla_xg$, we can similarly obtain
$$
\E\sup\limits_{t\in[0,T_0]}\left\|\nabla_x\int_0^t \mathcal{U}(t-s)(f_s-g_s)*\xi_{\delta}\, Q\dd W_s\right\|^p_{L^p_{x,v}} \leq C\, T_0\, \E\sup\limits_{s\in[0,T_0]}\|f_s-g_s\|^p_{L^{\infty}_{x,v}},
$$
where the constant $C$ now depends on $p$, $\kappa_{0,\infty}$, $\kappa_{1,\infty}$ and $\|\nabla_x\xi_{\delta}\|_{L^1_{x,v}}$. Furthermore, with the identity $\nabla_v\mathcal{U}(t)g=-ta'(v)\mathcal{U}(t)\nabla_xg+\mathcal{U}(t)\nabla_vg$, a similar bound can be proved for the derivatives of the stochastic integral with respect to $v\in V$. To sum up, we are led to 
$$
\E\sup\limits_{t\in[0,T_0]}\left\|\int_0^t \mathcal{U}(t-s)(f_s-g_s)*\xi_{\delta}\, Q\dd W_s\right\|^p_{W^{1,p}_{x,v}} \leq C\, (T_0+T_0^2)\, \E\sup\limits_{s\in[0,T_0]}\|f_s-g_s\|^p_{L^{\infty}_{x,v}},
$$
for some constant $C$ which depends on $p$, $\kappa_{0,\infty}$, $\kappa_{1,\infty}$, $\|\nabla_x\xi_{\delta}\|_{L^1_{x,v}}$ and $\|\nabla_v\xi_{\delta}\|_{L^1_{x,v}}$. Finally, with the Sobolev embedding $W^{1,p}_{x,v}\subset L^{\infty}_{x,v}$ which holds true since $p>N$, we can conclude that the contraction property of the stochastic term is satisfied in $L^p(\Omega;C([0,T_0];L^{\infty}_{x,v}))$ provided $T_0$ is sufficiently small. The Banach fixed-point Theorem then applies and gives us a mild solution $\frd$ of $(\ref{rtregultruncmild})$ in $L^p(\Omega;C([0,T_0];L^{\infty}_{x,v}))$. Iterating this argument yields a solution in the space $L^p(\Omega;C([0,T];L^{\infty}_{x,v}))$. Let us introduce, for $R>0$ and $\delta>0$, the following stopping times
$$\tau_{R,\delta}:=\inf\{t\in[0,T], \, \|\frd_t\|_{L^{\infty}_{x,v}}>R\}.$$
We can show, with a similar method as in \cite[Lemma 4.1]{debdeb}, that $\tau_{R,\delta}$ is nondecreasing with $R$ so that we can define $\tau^*_{\delta}:=\lim_{R\to\infty}\tau_{R,\delta}$. The next step is devoted to the proof of some estimates on the solution $\frd$.

{\em Step 3: Estimates on the solution $\frd$.} In this step, we emphasize the dependence through the parameters $R$ and $\delta$ of the constants $C$ appearing in the estimates. For instance $C_{\delta}$ depends on $\delta$ but not on $R$. With the mild formulation $(\ref{rtregultruncmild})$, using the boundedness of $\theta_R$, $\sigma$ and $G$, the contraction property of the semigroup $\mathcal{U}$ in $L^{\infty}_{x,v}$ and evaluating the stochastic integral in $L^{\infty}_{x,v}$ similarly as above, we can obtain the following bound
\begin{equation}\label{frdinfty}
\E\sup\limits_{t\in[0,T]}\|\frd_t\|^p_{L^{\infty}_{x,v}}\leq C_{\delta}.
\end{equation}
Note that the dependence with respect to $\delta$ of this bound is due to the evaluation of the stochastic integral in $L^{\infty}_{x,v}$ by estimating its $W^{1,p}_{x,v}$-norm: this gives rise to the terms $\|\nabla_x\xi_{\delta}\|_{L^1_{x,v}}$ and $\|\nabla_v\xi_{\delta}\|_{L^1_{x,v}}$ which depend on $\delta$. Nevertheless, estimating the solution $\frd$ in $L^p_{x,v}$  with $p>2$ gives a uniform bound with respect to $R$ and $\delta$. Precisely, with the mild formulation $(\ref{rtregultruncmild})$, using the boundedness of $\theta_R$, $\sigma$ and $G$, the contraction property of the semigroup $\mathcal{U}$ in $L^{p}_{x,v}$ and evaluating the stochastic integral in $L^{p}_{x,v}$, $p>2$, thanks to Lemma \ref{RegulStochastic}, we can obtain the following bound
\begin{equation}\label{frd2}
\E\sup\limits_{t\in[0,T]}\|\frd_t\|^p_{L^p_{x,v}}\leq C.
\end{equation}
Finally, we point out that we can also estimate $\nabla_x\frd$ in $L^p_{x,v}$, $p>2$, by differentiating equation $(\ref{rtregultruncmild})$. We obtain the bound 
\begin{equation}\label{dxfrd2}
\E\sup\limits_{t\in[0,T]}\|\nabla_x\frd\|^p_{L^p_{x,v}}\leq C_R.
\end{equation}

{\em Step 4: Definition of $\fd$.} From \eqref{frdinfty} we easily deduce that for all $\delta>0$, $\tau^*_{\delta}=T$ a.s. Thus, we define $\fd$ on $[0,T]=\cup_{R>0}[0,\tau_{R,\delta}]$ by $\fd = \frd$ on $[0,\tau_{R,\delta}]$. Note that this definition makes sense since we have proved uniqueness for the equation \eqref{rtregultruncmild} satisfied by $\frd$. Since $\frd$ is a mild solution of $(\ref{rtregultrunc})$ and since for all $t\in[0,T)$ we have that $\nabla_x\fd$ exists a.s. in $L^p(0,t;L^p_{x,v}))$, $p>2$, thanks to $(\ref{dxfrd2})$, we get that $\fd$ is a strong solution of $(\ref{rtregul})$, that is, $\PP-$a.s. for all $t\in[0,T]$,
\begin{equation}\label{fd}
\fd(t)  = \rho_{\text{in}}-\int_0^{t}a(v)\cdot\nabla_x\fd_s\, \dd s+\int_0^{t}\sigma(\li{\fd_s})L(\fd_s)\, \dd s + \int_0^{t}G\fd_s\, \dd s + \int_0^{t}\fd_s*\xi_{\delta}\, Q\dd W_s.
\end{equation}
Furthermore, with \eqref{frd2} and the fact that $\tau^*_{\delta}=T$ a.s., we deduce that for $p>2$,
$$
\E\sup\limits_{t\in[0,T]}\|\fd_t\|^p_{L^p_{x,v}}\leq C.
$$
Thanks to the H\"older inequality, the previous bound holds true when $p=2$, that is
\begin{equation}\label{fd2}
\E\sup\limits_{t\in[0,T]}\|\fd_t\|^2_{L^2_{x,v}}\leq C.
\end{equation}
Finally, note that, thanks to the equation \eqref{fd}, we can show that $\fd\geq 0$. Indeed, it suffices to apply the It\^o formula with the function $\varphi_{\delta}$ defined by $(\ref{defvarphi})$ to the process $-\fd$. Similarly as in \textit{Step 1}, since $\rho_{\text{in}}\geq 0$, this yields $(\fd)^-=0$, hence the result.

{\em Step 5: Convergence $\delta\to 0$.} Thanks to \eqref{fd2}, up to a subsequence, the sequence $(\fd)_{\delta>0}$ converges weakly in $L^2(\Omega;L^2(0,T;L^2_{x,v}))$ to some $f$. This is not sufficient to pass to the limit in \eqref{fd} due to the non-linear term. Thus we use the following stochastic averaging lemma, the proof of which in given in Appendix B.
\begin{lemma}\label{lemmemoyenne}
Let $\alpha \in (0,1]$. We assume that hypothesis \eqref{nondegenlemmemoy} is satisfied. Let $f$ be bounded in $L^2(\Omega;L^2(0,T;L^2_{x,v}))$ such that
\begin{equation}
\dd f + a(v)\cdot\nabla_x f \dd t = h \dd t + g\, Q\dd W_t,
\end{equation}
with $g$ and $h$ bounded in $L^2(\Omega;L^2(0,T;L^2_{x,v}))$. Then the quantity $\rho = \li{f}$ verifies
$$\E\int_0^T\|\rho_s\|^2_{H^{\alpha/2}_x}\dd s\leq C.$$
\end{lemma}
\noindent With \eqref{fd} and \eqref{fd2}, we apply this lemma to the process $\rho^{\delta}:=\li{\fd}$ to obtain
\begin{equation}\label{tight1}
\E\int_0^T\|\rho^{\delta}_s\|^2_{H^{\alpha/2}_x}\dd s\ \leq\ C.
\end{equation}
Furthermore, thanks to \eqref{fd} and \eqref{fd2}, we get that
\begin{equation}\label{tight2}
\E\int_0^{T-h}\|\fd_{s+h}-\fd_s\|^2_{H^{-1}_{x,v}}\dd s\ \leq\ Ch,
\end{equation}
which also implies
\begin{equation}\label{tight3}
\E\int_0^{T-h}\|\rho^{\delta}_{s+h}-\rho^{\delta}_s\|^2_{H^{-1}_{x}}\dd s\ \leq\ Ch.
\end{equation}
Then, with the bounds \eqref{fd2} and \eqref{tight2} and \cite[Theorem 1]{simon} we obtain that the sequence of the laws of the processes $(\fd)_{\delta>0}$ is tight in $L^2(0,T;H^{-1}_{x,v})$. With the bounds \eqref{tight1} and \eqref{tight3} and \cite[Theorem 4]{simon} we also get that the sequence of the laws of the processes $(\rho^{\delta})_{\delta>0}$ is tight in $L^2(0,T;L^2_{x,v})$. As a consequence, with Prokhorov's Theorem, we can assume that, up to a subsequence, the laws of the processes $(\rho^{\delta})_{\delta>0}$ converges weakly to the law of some process $\rho$ in the space of probability measures on $L^2(0,T;L^2_{x,v})$. Then, using then the Skorohod representation Theorem, there exist a new probability space $(\ti{\Omega},\ti{\mathcal{F}},\ti{\PP})$ where lives a cylindrical Wiener process $\ti{W}$ on the Hilbert space $L^2(\T^N)$ and some random variables $\ti{\fd}$, $\ti{f}$ with respective laws $\PP(\fd\in\cdot)$ and $\PP(f\in\cdot)$ such that $\int_V\ti{\fd}\,\dd v$ converges $\ti{\PP}-$a.s. in $L^2(0,T;L^2_{x,v})$ to $\int_V\ti{f}\,\dd v$. Furthermore, we recall that we have the weak convergence of $\ti{\fd}$ to $\ti{f}$ in $L^2(\ti{\Omega};L^2(0,T;L^2_{x,v}))$. We also point out that, with \eqref{dxfrd2}, we can suppose that $\nabla_x\ti{f}$ exists a.s. in $L^2(0,t;L^2_{x,v})$ for all $t\in[0,T)$. We now have all in hands to pass to the limit $\delta\to 0$ in \eqref{fd} to discover that $\ti{\PP}-$a.s. for all $t\in[0,T]$,
\begin{equation}\label{frdchange}
\ti{f}(t)=\rho_{\text{in}}-\int_0^t a(v)\cdot\nabla_x\ti{f_s}\, \dd s+\int_0^{t}\sigma(\li{\ti{f_s}})L(\ti{f_s})\, \dd s + \int_0^{t} G\ti{f_s}\, \dd t + \int_0^{t}\ti{f_s}\, Q\dd \ti{W}_s.
\end{equation}

{\em Step 6: Conclusion.} In this final step, we want to get rid of the change of probability space. To this purpose, we recall that we proved pathwise uniqueness for positive solutions to the equation $(\ref{frdchange})$ above in \textit{Step 1}. As a consequence, we will make use of the Gy\"{o}ngy-Krylov characterization of convergence in probability introduced in \cite{GK}. We recall here the precise result
\begin{lemma}
Let $X$ be a Polish space equipped with the Borel $\sigma$-algebra. A sequence of $X$-valued random variables $\{Y_n,\,n\in\N\}$ converges in probability if and only if for every subsequence of joint laws $\{\mu_{n_k,m_k},\,k\in\N\}$, there exists a further subsequence which converges weakly to a probability measure $\mu$ such that 
$$\mu\left((x,y)\in X\times X,\, x=y\right)=1.$$
\end{lemma}
\noindent Thanks to the pathwise uniqueness of equation $(\ref{frdchange})$, we can make use of this characterization of convergence in probability here (see for instance \cite[Proof of Theorem 2.1]{Gyongy} for more details about the arguments) to deduce that, up to a subsequence, the sequence $(\fd)_{\delta>0}$ defined on the initial probability space $(\Omega,\mathcal{F},\PP)$ converges in probability in $L^2(0,T;L^2_{x,v})$ to a process $f$. Without loss of generality, we can assume that the convergence is almost sure. Then, using again the method used above in \textit{Step 5}, we deduce that $\PP-$a.s. for all $t\in[0,T]$,
\begin{equation}
f(t)=\rho_{\text{in}}-\int_0^{t}a(v)\cdot\nabla_xf_s\, \dd s+\int_0^{t}\sigma(\li{f_s})L(f_s)\, \dd s + \int_0^{t} Gf_s\, \dd t + \int_0^{t}f_s\, Q\dd W_s.
\end{equation}
Thus $f$ is a non-negative strong solution of the kinetic problem $(\ref{rt})$ and belongs to the expected spaces.

{\em Uniform bound part.} The bound $(\ref{unifbound})$ is easily obtained with an application of the It\^o formula with the $C^2$ function $f\mapsto\|f\|^2_{L^2_{x,v}}$ to the process $\fe$. We then make use of the dissipation property $(\ref{sigmaLL2})$ of the operator $\sigma(\li{\cdot}) L(\cdot)$ in $L^2_{x,v}$ and of the Gronwall lemma.
\end{proof}

\subsection{Existence and regularity for the limiting equation}

\noindent Let us now study the limiting stochastic fluid equation $(\ref{EqF})$ and the regularity of its solution. Precisely, we have the following result.
\begin{proposition}\label{existencerho}
Let $p\geq 1$. There exists a strong solution $\rho$ which belongs to $L^p(\Omega;C^{0,3}([0,T]\times\T^N))$ to the limit equation $(\ref{EqF})$
$$
\left\{
\begin{aligned}
& \dd \rho - \mathrm{div}_x\left(\sigma(\rho)^{-1}\li{K}\,\nabla_x \rho \right) \dd t = \rho\circ Q\dd W_t,\\
& \rho(0)=\rho_{\text{in}},
\end{aligned}
\right.
$$
that is, $\PP-$a.s. for all $t\in[0,T]$,
$$\rho(t)=\rho_{\text{in}}+\int_0^t\mathrm{div}_x\left(\sigma(\rho)^{-1}\li{K}\,\nabla_x \rho \right) \dd s+\int_0^t\rho\circ Q\dd W_s.$$
\end{proposition}
\begin{proof}Note that the Stratonovich integral $\rho\circ Q\dd W_t$ rewrites in It\^o form $G\rho\,\dd t + \rho\,Q\dd W_t$. As a consequence, with the hypothesis made on $\sigma$ (H1)$-$(H3), $a$, and the noise $(\ref{regularnoise})$, we can easily show that the results of \cite{regularity} apply so that the proof is complete.
\end{proof}

\subsection{Definition of the two first correctors}\label{defcorrec}

\noindent Following the computations done in a formal way in section \ref{sec:formalhilbert}, we define:
\begin{equation}\label{deff12}
\begin{aligned}
f_1 &:=-\sigma(\rho)^{-1}\,a(v)\cdot\nabla_x \rho,\\
f_2 &:= -\sigma(\rho)^{-1}\mathrm{div}_x\left(\sigma(\rho)^{-1}(\li{K}-K)\nabla_x \rho  \right).
\end{aligned}
\end{equation}
We state two propositions giving the properties of the processes $f_1$ and $f_2$.

\begin{proposition}\label{propf1}
Let $p\geq 1$. The first corrector $f^1$, defined by $(\ref{deff12})$, satisfies
\begin{equation}\label{f1eq}
\sigma(\rho)L(f_1) = a(v)\cdot\nabla_x \rho
\end{equation}
with the estimate
\begin{equation}\label{f1infty}
\E\sup\limits_{t\in[0,T]}\|f_1(t)\|^p_{L^{\infty}_{x,v}}<\infty, \quad \E\sup\limits_{t\in[0,T]}\|a(v)\cdot\nabla_xf_1(t)\|^p_{L^{\infty}_{x,v}}<\infty.
\end{equation}
Furthermore, we have
\begin{equation}\label{df1eq}
\dd f_1 = f_{1,d}\ \dd t + f_1\left(1-\sigma(\rho)^{-1}\sigma'(\rho)\rho\right)Q\dd W_t,
\end{equation}
where $f_{1,d}$ satisfies
\begin{equation}\label{D1infty}
\E\sup\limits_{t\in[0,T]}\|f_{1,d}(t)\|^p_{L^{\infty}_{x,v}}<\infty.
\end{equation}
\end{proposition}
\begin{proof}
The equation $(\ref{f1eq})$ is a straightforward consequence of the definition of $L$ and $f_1$ and of $(\ref{nullflux})$. The estimate $(\ref{f1infty})$ is a consequence of the regularity of $\rho$ given in Proposition \ref{existencerho}, the bounds (H1) on $\sigma$ and the boundedness of $a$. One can easily verify that equation $(\ref{df1eq})$ holds true; then the bound $(\ref{D1infty})$ comes once again from the regularity of $\rho$, the bounds (H1) on $\sigma$, the regularity (H2) of $\sigma$ and the boundedness of $a$.
\end{proof}
\noindent Similarly, we can prove the following properties of the second corrector $f_2$.
\begin{proposition}\label{propf2}
Let $p\geq 1$. The second corrector $f^2$, defined by $(\ref{deff12})$, satisfies
\begin{equation}\label{f2eq}
\sigma(\rho)L(f_2) = \mathrm{div}_x\left(\sigma(\rho)^{-1}(\li{K}-K)\nabla_x \rho\right)=\mathrm{div}_x\left(\sigma(\rho)^{-1}\li{K}\,\nabla_x \rho\right)+a(v)\cdot\nabla_x f_1
\end{equation}
with the estimates
\begin{equation}\label{f2infty}
\E\sup\limits_{t\in[0,T]}\|f_2(t)\|^p_{L^{\infty}_{x,v}}<\infty,\quad \E\sup\limits_{t\in[0,T]}\|a(v)\cdot\nabla_xf_2(t)\|^p_{L^{\infty}_{x,v}}<\infty.
\end{equation}
Furthermore, we have 
\begin{equation}\label{df2eq}
\dd f_2 = f_{2,d}\ \dd t + f_{2,s}\ Q\dd W_t,
\end{equation}
where $f_{2,d}$ and $f_{2,s}$ satisfy
\begin{equation}\label{D2infty}
\E\sup\limits_{t\in[0,T]}\|f_{2,d}(t)\|^p_{L^{\infty}_{x,v}}<\infty,\quad \E\sup\limits_{t\in[0,T]}\|f_{2,s}(t)\|^p_{L^{\infty}_{x,v}}<\infty.
\end{equation}
\end{proposition}

\subsection{Equation satisfied by the remainder}\label{formalre}

\noindent From now on, $\fe$ denotes the solution to problem $(\ref{rt})$ and $\rho$ the solution of the limiting equation $(\ref{EqF})$. We define the remainder $\re$ by
$$\re := \fe - \rho - \eps f_1 - \eps^2 f_2 - \eps^3 \feth,$$
where the correctors $f_1$, $f_2$ have been defined above. The third corrector $\feth$ will be defined below; its aim will be to cancel all the noise terms of order $O(\eps)$ so that the remainder has a noise term of order $O(\eps^2)$.
Let us write the equation satisfied by $\re$.  We have 
$$
\dd \re  =\ -\frac{1}{\eps}a(v)\cdot\nabla_x \fe\ \dd t + \frac{1}{\eps^2}\sigma(\li{\fe})L(\fe)\ \dd t + \fe\ Q\dd W_t + G\fe\ \dd t -\dd \rho - \eps \dd f_1 - \eps^2 \dd f_2 - \eps^3 \dd \feth.
$$
We recall that $L(\rho)=0$ so that we have
$$
\begin{aligned}
\dd \re & =\ -\frac{1}{\eps}a(v)\cdot\nabla_x \fe\ \dd t + \frac{1}{\eps}\sigma(\rho)L(f_1)\ \dd t + \sigma(\rho)L(f_2)\ \dd t + \eps\sigma(\rho)L(\feth)\ \dd t \\
& +\ \frac{1}{\eps^2}\left[\sigma(\li{\fe})L(\fe)-\sigma(\rho)L(\fe-\re)\right] \dd t \\
& +\ \fe\ Q\dd W_t + G\fe\ \dd t-\dd \rho - \eps \dd f_1 - \eps^2 \dd f_2 - \eps^3 \dd \feth.
\end{aligned}
$$
Using the equations satisfied by $f_1$, $f_2$ and $\rho$, that is $(\ref{f1eq})$, $(\ref{f2eq})$ and $(\ref{EqF})$, we obtain
$$
\begin{aligned}
\dd \re & =\ -\frac{1}{\eps}a(v)\cdot\nabla_x \fe\ \dd t + \frac{1}{\eps}a(v)\cdot\nabla_x\rho\ \dd t  + a(v)\cdot\nabla_x f_1\ \dd t + \mathrm{div}_x\left(\li{K}\sigma(\rho)^{-1}\,\nabla_x \rho \right)\dd t \\
& + \eps\sigma(\rho)L(\feth)\ \dd t + \frac{1}{\eps^2}\left[\sigma(\li{\fe})L(\fe)-\sigma(\rho)L(\fe-\re)\right] \dd t \\ 
& +\ \fe\ Q\dd W_t+G\fe\ \dd t-\mathrm{div}_x\left(\li{K}\sigma(\rho)^{-1}\,\nabla_x \rho \right)\dd t - \rho\ Q\dd W_t - G\rho\ \dd t \\
&-\ \eps \dd f_1 - \eps^2 \dd f_2 - \eps^3 \dd \feth.
\end{aligned}
$$
After simplification, we have,
$$
\begin{aligned}
\dd \re +\frac{1}{\eps}a(v)\cdot\nabla_x\re\ \dd t &=\ -\eps a(v)\cdot\nabla_x f_2\ \dd t -\eps^2 a(v)\cdot\nabla_x \feth\ \dd t\\
& + \frac{1}{\eps^2}\left[\sigma(\li{\fe})L(\fe)-\sigma(\rho)L(\fe-\re)\right]\dd t + (\fe-\rho)\, Q\dd W_t\\
&  +G(\fe-\rho)\,\dd t  - \eps \dd f_1 - \eps^2 \dd f_2 - \eps^3 \dd \feth  + \eps\sigma(\rho)L(\feth)\ \dd t.
\end{aligned}
$$
Using the expression $(\ref{df1eq})$ of $\dd f_1$, we discover
$$
\begin{aligned}
\dd \re +\frac{1}{\eps}a(v)\cdot\nabla_x\re \dd t\ & =\ -\eps a(v)\cdot\nabla_x f_2\ \dd t -\eps^2 a(v)\cdot\nabla_x \feth\ \dd t  \\
& + \frac{1}{\eps^2}\left[\sigma(\li{\fe})L(\fe)-\sigma(\rho)L(\fe-\re)\right]\dd t + (\fe-\rho)\, Q\dd W_t \\ 
&  +G(\fe-\rho)\,\dd t - \eps f_{1,d}\ \dd t - \eps f_1\left(1-\sigma(\rho)^{-1}\sigma'(\rho)\rho\right)Q\dd W_t  \\
&  -\ \eps^2 \dd f_2 - \eps^3 \dd \feth  + \eps\sigma(\rho)L(\feth)\ \dd t.
\end{aligned}
$$
In the sequel, when estimating the remainder, we need the noise term to be of order $O(\eps^2)$, see Section \ref{sec:estimatere}. As a consequence, we would like to choose $\feth$ to delete the terms of order $O(\eps)$ in front of the noise. Namely, we would like to impose
\begin{equation}\label{formalf3}
\eps^2\dd\feth-\sigma(\rho)L(\feth)\ \dd t=f_1\sigma(\rho)^{-1}\sigma'(\rho)\rho\ Q\dd W_t,
\end{equation}
so that the equation satisfied by the remainder $\re$ is finally given by
$$
\begin{aligned}
\dd \re +\frac{1}{\eps}a(v)\cdot\nabla_x\re \dd t & =\ -\eps a(v)\cdot\nabla_x f_2\ \dd t -\eps^2 a(v)\cdot\nabla_x \feth\ \dd t  \\
& + \frac{1}{\eps^2}\left[\sigma(\li{\fe})L(\fe)-\sigma(\rho)L(\fe-\re)\right]\dd t \\
& +\ (\fe-\rho-\eps f_1)\, Q\dd W_t + G(\fe-\rho)\,\dd t  -\eps f_{1,d}\ \dd t - \eps^2 \dd f_2.
\end{aligned}
$$
Note that $f_1$ and $f_2$ do not depend on $\eps$. In the following, we shall prove that $\feth$ is of order $O(\eps^{-1})$ with respect to $\eps$. As a consequence, the drift term (excepted the singular one) is of order $O(\eps)$. We also recall that we precisely added $\feth$ in the development of $\fe$ to get a term of order $O(\eps^2)$ in front of the noise; this will be necessary further in the estimate of the remainder. We point out that $L^1_{x,v}$ is indeed the appropriate space in which the estimate of the remainder will give a favourable sign to the singular term $\eps^{-2}\left[\sigma(\li{\fe})L(\fe)-\sigma(\rho)L(\fe-\re)\right]$ thanks to the accretivity property of the operator $\sigma(\li{\cdot}) L(\cdot)$, see section \ref{sec:sigmaL}. The next section is devoted to the definition of the third corrector by solving the equation $(\ref{formalf3})$.

\subsection{Definition of the third corrector.}

\noindent In this part, we study the following equation for the third corrector which was suggested in a formal way in the computations done just above:
\begin{equation}\label{f3eq}
\eps^2\dd\feth-\sigma(\rho)L(\feth)\ \dd t=f_1\sigma(\rho)^{-1}\sigma'(\rho)\rho\ Q\dd W_t.
\end{equation}
We solve this equation thanks to a stochastic convolution with the semigroup generated by the \textit{non-autonomous} operator $\sigma(\rho)L$ on $L^p_{x,v}$ where $p\geq 1$. Let us begin with the study of this semigroup. We point out that we only need to know its behaviour on the subspace $\{g\in L^p_{x,v},\,\li{g}=0\}$.
\begin{proposition}\label{semigroup}
Let $p\geq 1$ and $g\in L^p_{x,v}$ such that $\li{g}=0$. For $s\in[0,T]$, the problem
\begin{equation}\label{semigr}
\left\{
\begin{array}{l}
\eps^2 u'(t) - \sigma(\rho(t))L(u(t)) = 0, \quad t\in[s,T], \\
u(s)=g,
\end{array}
\right.
\end{equation}
admits a.s. a unique classical solution in $\mathcal{C}^1([s,T];L^p_{x,v})$ that we write $u(t)=U^{\eps}(t,s)g$. It is given by
\begin{equation}\label{solsemigr}
U^{\eps}(t,s)g= g \exp\left(-\eps^{-2}\int_s^t\sigma(\rho)(r)\,\dd r \right), \quad t\in[s,T].
\end{equation}
Furthermore, we have the bound
\begin{equation}\label{decrsemigr}
\|U^{\eps}(t,s)g\|_{L^p_{x,v}}\leq \|g\|_{L^p_{x,v}}\exp\left(-\eps^2\sigma_*(t-s)\right).
\end{equation}
\end{proposition}
\begin{proof}
Note that with $(\ref{semigr})$ and $\int_V\sigma(\rho)L(u)\,\dd v=0$, we immediately have that $\li{u}'=0$ so that $\li{u}$ is constant and equals $\li{g}$, which is zero. Then equation $(\ref{semigr})$, with the definition of $L$, rewrites
$$\eps^2 u'(t) = - \sigma(\rho(t))u(t),$$
which gives easily $(\ref{solsemigr})$. This proves existence and uniqueness in $\mathcal{C}^1([0,T];L^p_{x,v})$ for the problem $(\ref{semigr})$. The bound $\sigma\geq \sigma_*$ (H1) immediately yields $(\ref{decrsemigr})$. This concludes the proof.
\end{proof}
\noindent Before stating the main result about the third corrector, we need the following lemma about the regularity of the stochastic convolution.
\begin{lemma}\label{regulsto}
Let $p\geq 2$. Suppose that $\varphi\in L^p(\Omega;L^{\infty}(0,T;L^p_{x,v}))$ satisfies $\li{\varphi}=0$. We define 
$$z(t):=\eps^{-2}\int_0^tU^{\eps}(t,s)\varphi(s)\,Q\dd W_s, \quad t\in[0,T].$$
Then, we have the bound
$$\sup\limits_{t\in[0,T]}\E\|z(t)\|^p_{L^p_{x,v}}\leq C \eps^{-p}\, \E\|\varphi\|^p_{L^{\infty}(0,T;L^p_{x,v})}.$$
\end{lemma}
\begin{proof}
Here, we recall that a.s. and for $s,\,t\in[0,T]$, $U^{\eps}(t,s)\varphi(s)$ is an element of $L^p_{x,v}$. The stochastic integral $U^{\eps}(t,s)\varphi(s)Q\dd W_s$ can be interpreted as $\Psi^{\eps}(t,s)\dd W_s$ where $\Psi^{\eps}(t,s)$ is the following $\gamma$-radonifying operator from $L^2(\T^N)$ to $L^p_{x,v}$ (see Subsection \ref{sec:noise})
$$\Psi^{\eps}(t,s)(e_k):=U^{\eps}(t,s)\varphi(s)Qe_k.$$ Then, we use the Burkholder-Davis-Gundy's inequality for martingales with values in $L^p_{x,v}$ (see \cite{gamma1} and \cite{gamma2}) and the bound $(\ref{gammanorm})$ to obtain 
$$
\begin{aligned}
\E\|z(t)\|^p_{L^p_{x,v}} &\leq\ C\eps^{-2p}\, \E\left(\int_0^t\left\|\Psi^{\eps}(t,s)\right\|^2_{\gamma(L^2_x,L^p_{x,v})}\,\dd s\right)^{p/2} \\
&\leq\ C\eps^{-2p}\, \E\left(\int_0^t\left\|U^{\eps}(t,s)\varphi(s)\right\|^2_{L^p_{x,v}}\,\dd s\right)^{p/2}. \\
\end{aligned}
$$
Next, thanks to $(\ref{decrsemigr})$ with the hypothesis $\li{\varphi}=0$, we have
$$
\begin{aligned}
\E\|z(t)\|^p_{L^p_{x,v}} &\leq\ C\eps^{-2p}\, \E\|\varphi\|^p_{L^{\infty}(0,T;L^p_{x,v})} \left(\int_0^t\exp\left(-\eps^{-2}\sigma_*(t-s)\right)\,\dd s\right)^{p/2} \\
&\leq\ C \eps^{-p}\, \E\|\varphi\|^p_{L^{\infty}(0,T;L^p_{x,v})},
\end{aligned}
$$
which concludes the proof.
\end{proof}
\noindent The existence and the properties of the third corrector $\feth$ are collected in the following proposition.
\begin{proposition}\label{propf3}
Let $p\geq 1$. There exists a process $\feth$ with values in $L^{\infty}(0,T;L^p(\Omega;L^p_{x,v}))$ which satisfies $\li{\feth}=0$ and
\begin{equation}\label{f3eqbis}
\eps^2\dd\feth-\sigma(\rho)L(\feth)\ \dd t=f_1\sigma(\rho)^{-1}\sigma'(\rho)\rho\ Q\dd W_t,
\end{equation}
that is, $\PP-$a.s. for all $t\in[0,T]$,
$$
\feth(t) = \eps^{-2}\int_0^t\sigma(\rho(s))L(\feth(s))\, \dd s + \eps^{-2}\int_0^tf_1(s)\sigma(\rho(s))^{-1}\sigma'(\rho(s))\rho(s)\, Q\dd W_s.
$$
Furthermore, we have the estimates
\begin{equation}\label{f3infty}
\sup\limits_{t\in[0,T]}\E\|f_3(t)\|^p_{L^p_{x,v}}\leq C\eps^{-p},\quad \sup\limits_{t\in[0,T]}\E\|a(v)\cdot\nabla_xf_3(t)\|^p_{L^p_{x,v}}\leq C\eps^{-p}.
\end{equation}
\end{proposition}
\begin{proof}
We fix $p\geq 2$. We set $\varphi:=f_1\sigma(\rho)^{-1}\sigma'(\rho)\rho$ and we define
$$\feth(t):=\eps^{-2}\int_0^tU^{\eps}(t,s)\varphi(s)\,Q\dd W_s, \quad t\in[0,T].$$
Observe that with the definition $(\ref{deff12})$ of $f_1$ we have $\varphi=-\sigma(\rho)^{-2}\,\sigma'(\rho)\rho\,a(v)\cdot\nabla_x\rho$. Thanks to the regularity of $\rho$, $\sigma$ and $a$, we obviously have that $\varphi$ belongs to $L^p(\Omega;L^{\infty}(0,T;L^{\infty}_{x,v}))$ which is embedded in $L^p(\Omega;L^{\infty}(0,T;L^p_{x,v} ))$. As a consequence, since $\li{\varphi}=0$, we can apply Lemma $\ref{regulsto}$ to find that 
\begin{equation}\label{Lpbound}
\sup\limits_{t\in[0,T]}\E\|\feth(t)\|^p_{L^p_{x,v}} \leq C \eps^{-p}.
\end{equation}
This proves in particular the existence of the stochastic integral which defines $\feth$. Next, for $t\in[0,T]$, we can easily compute the quantity 
$$\int_0^t\sigma(\rho(s))L(\feth(s))\,\dd s,$$ 
by using the stochastic version of Fubini's Theorem and the fact that, when $g\in L^p_{x,v}$, $\d_s U^{\eps}(s,r)=\eps^{-2}\sigma(\rho(s))L(U^{\eps}(s,r)g)$ by Proposition \ref{semigr}; we obtain that $\feth$ is a strong solution of $(\ref{f3eqbis})$, that is, $\PP-$a.s.,
\begin{equation}\label{strong}
\feth(t) = \eps^{-2}\int_0^t\sigma(\rho(s))L(\feth(s))\, \dd s + \eps^{-2}\int_0^t\varphi(s)\, Q\dd W_s.
\end{equation}
With this expression of $\feth$, it is clear that $\li{\feth}=0$. To conclude the proof, it remains to bound $a(v)\cdot\nabla_x\feth$ in $L^{\infty}(0,T;L^p(\Omega;L^p_{x,v}))$. Let $i\in\{1,...,N\}$, we differentiate equation $(\ref{strong})$ with respect to the space variable $x_i$ to discover
$$
\begin{aligned}
\d_{x_i}\feth(t) &=\ \eps^{-2}\int_0^t\d_{x_i}\rho_s\sigma'(\rho(s))L(\feth(s))\, \dd s + \eps^{-2}\int_0^t\sigma(\rho(s))L(\d_{x_i}\feth(s))\, \dd s \\
& +\ \eps^{-2}\int_0^t\d_{x_i}\varphi(s)\, Q\dd W_s + \eps^{-2}\int_0^t\varphi(s)\, Q\dd \d_{x_i}W_s. 
\end{aligned}
$$
As a consequence, we see that we can write $\d_{x_i}\feth$ into the following mild form
$$
\begin{aligned}
\d_{x_i}\feth(t) &=\ \eps^{-2}\int_0^tU^{\eps}(t,s)\d_{x_i}\rho_s\sigma'(\rho(s))L(\feth(s))\, \dd s + \eps^{-2}\int_0^tU^{\eps}(t,s)\d_{x_i}\varphi(s)\, Q\dd W_s \\
& +\ \eps^{-2}\int_0^tU^{\eps}(t,s)\varphi(s)\, Q\dd \d_{x_i}W_s. 
\end{aligned}
$$
Let us deal with the first term of the last equality. We set $\phi=\d_{x_i}\rho_s\sigma'(\rho(s))L(\feth(s))$. Thanks to the regularity of $\rho$, $\sigma$ and the bound $(\ref{Lpbound})$, it clearly belongs to the space $L^p(\Omega;L^p(0,T;L^p_{x,v}))$ with
$$\E\|\phi\|^p_{L^p(0,T;L^p_{x,v})}\leq C \eps^{-p}.$$
Therefore, since $\li{\phi}=0$, we can use $(\ref{decrsemigr})$ to write, with the Young and H\"{o}lder inequalities,
$$
\begin{aligned}
\E\left\|\eps^{-2}\int_0^tU^{\eps}(t,s)\phi(s)\dd s\right\|^p_{L^p_{x,v}} &\leq\ \E\left(\int_0^t\eps^{-2}\|U^{\eps}(t,s)\phi(s)\|_{L^p_{x,v}}\dd s\right)^p\\
& \leq\ C\E\|\phi\|^p_{L^p(0,T;L^p_{x,v})}\leq C\eps^{-p}.
\end{aligned}
$$
For the two remaining terms, we can easily verify that Lemma \ref{regulsto} applies (even with the noise $\dd \d_{x_i}W$ thanks to the hypothesis \eqref{regularnoise01} $\sum_{k}\|\d_{x_i}Qe_k\|^2_{\infty}<\infty$). Finally, we combine the two applications of Lemma \ref{regulsto} with the previous bound to obtain
$$
\sup\limits_{t\in[0,T]}\E\|\nabla_x\feth(t)\|^p_{L^p_{x,v}} \leq C \eps^{-p},
$$
which concludes the proof of the second estimate of $(\ref{f3infty})$ due to the boundedness of $a$. It remains to prove the proposition when $p\in[1,2)$ but it is a straightforward consequence of the H\"{o}lder's inequality. This concludes the proof.
\end{proof}

\subsection{Estimate of the remainder}\label{sec:estimatere}

\noindent Finally, we estimate the remainder $\re$ in the space $L^1_{x,v}$; this will conclude the proof of Theorem \ref{mainresult}. We point out that the correctors $f_1$, $f_2$ and $\feth$ are now properly defined in the previous sections. We recall that we set:
$$\re := \fe - \rho - \eps f^1 - \eps^2 f_2 - \eps^3 \feth.$$
Thanks to the calculations made in Subsection \ref{formalre}, $\re$ now satisfies:
$$
\begin{aligned}
\dd \re +\frac{1}{\eps}a(v)\cdot\nabla_x\re \dd t & =\ -\eps a(v)\cdot\nabla_x f_2\ \dd t -\eps^2 a(v)\cdot\nabla_x \feth\ \dd t  \\
& + \frac{1}{\eps^2}\left[\sigma(\li{\fe})L(\fe)-\sigma(\rho)L(\fe-\re)\right]\dd t \\
& +\ (\fe-\rho-\eps f_1)\, Q\dd W_t + G(\fe-\rho)\,\dd t  -\eps f_{1,d}\ \dd t - \eps^2 \dd f_2.
\end{aligned}
$$
We will estimate $\re$ in $L^1_{x,v}$ by estimating $(\re)^+$ and $(\re)^-$ in $L^1_{x,v}$ using the It\^o formula, where $x^+=\max(0,x)$ and $x^-=(-x)^+$. We write the equation verified by $\re$ as follows:
$$
\dd \re + \frac{1}{\eps}a(v)\cdot\nabla_x\re =\ D_t\ \dd t + \frac{1}{\eps^2}D^*_t\ \dd t  + H_t\ Q\dd W_t,
$$
where 
$$\left\{
\begin{aligned}
&D := -\eps a(v)\cdot\nabla_x f_2 -\eps^2 a(v)\cdot\nabla_x \feth + G(\fe-\rho)  -\eps f_{1,d} -\eps^2f_{2,d}, \\
&D^*\!\!:= \sigma(\li{\fe})L(\fe)-\sigma(\rho)L(\fe-\re), \\
&H := (\fe-\rho-\eps f_1) - \eps ^2 f_{2,s}.
\end{aligned}
\right.
$$
Since $\fe-\rho = \eps f^1 + \eps^2 f_2 + \eps^3 \feth + \re$, thanks to $(\ref{f1infty})$, $(\ref{D1infty})$, $(\ref{f2infty})$, $(\ref{D2infty})$, $(\ref{f3infty})$  with $p=1$ and with $\|G\|_{L^{\infty}}<\infty$, we have the bound
\begin{equation}\label{D}
\E\int_0^t\int_{\T^N\times V}|D_s|\,\dd s\ \leq\ C\eps + \E\int_0^t\int_{\T^N\times V}|\re_s|\,\dd s.
\end{equation}
Similarly, for any $\delta>0$, with $\fe-\rho - \eps f^1 = \eps^2 f_2 + \eps^3 \feth + \re$ and thanks to $(\ref{f2infty})$, $(\ref{D2infty})$, $(\ref{f3infty})$ with $p=2$ and with $\|G\|_{L^{\infty}}<\infty$, we have the bound
\begin{equation}\label{H}
\E\int_0^t\int_{\T^N\times V}G|H_s|^2\mathbf{1}_{|\re_s|\leq \delta}\,\dd s\ \leq\ C(\eps^4+\delta^2).
\end{equation}
Now, $\delta>0$ being fixed, we apply the It\^o formula with the $C^2$ approximation $\varphi_{\delta}$ of the function $x\mapsto x^+$ defined by $(\ref{defvarphi})$ to the process $\re$ to obtain (note that the term relative to $\eps^{-1}a(v)\cdot\nabla_x\re$ cancels)
$$
\begin{aligned}
\E\int_{\T^N\times V}\varphi_{\delta}(\re_t)\ & =\ \E\int_{\T^N\times V}\varphi_{\delta}(\re_{\text{in}})\ +\ \E\int_0^t\int_{\T^N\times V}\varphi_{\delta}'(\re_s)D_s\, \dd s \\
& +\ \frac{1}{\eps^2}\E\int_0^t\int_{\T^N\times V}\varphi_{\delta}'(\re_s)D^*_s\, \dd s\ +\ \E\int_0^t\int_{\T^N\times V}\varphi_{\delta}''(\re_s)G|H_s|^2\,\dd s.
\end{aligned}
$$
Since $x^+\leq\varphi_{\delta}(x)+\delta$, we have
$$\E\|(\re_t)^+\|_{L^1_{x,v}}\leq\ \E\int_{\T^N\times V}\varphi_{\delta}(\re_t) + \delta$$
and thanks to $\varphi_{\delta}(x)\leq x^+ $, we get
$$\E\int_{\T^N\times V}\varphi_{\delta}(\re_{\text{in}})\ \leq\ \E\|(\re_{\text{in}})^+\|_{L^1_{x,v}}.$$
With $|\varphi_{\delta}'|\leq 1$ and $(\ref{D})$, we have
$$\E\int_0^t\int_{\T^N\times V}\varphi_{\delta}'(\re_s)D_s\,\dd s\ \leq\ C\eps + \E\int_0^t\|\re_s\|_{L^1_{x,v}}\dd s.$$
Next, we study the term 
$$\int_{V}\varphi_{\delta}'(\re_s)D^*_s \,\dd v = \int_{V}\varphi_{\delta}'(\re_s)\left[\sigma(\li{\fe_s})L(\fe_s)-\sigma(\rho_s)L(\fe_s-\re_s)\right]\dd v.$$ 
To this end, we define $\ve := \fe-\re$; note that $\li{\ve}=\rho$. The term we are interested in thus rewrites
$$J:=\int_{V}\varphi_{\delta}'(\fe-\ve)\left[\sigma(\li{\fe})L(\fe)-\sigma(\li{\ve})L(\ve)\right]\dd v,$$
so that, with the positivity of $\fe$, we can apply the accretivity bound $(\ref{accr1})$ of Lemma \ref{accr} to find
$$J\leq C(1+\|\fe\|_{L^1_v})\delta.$$
We immediately deduce, using Cauchy-Schwarz's inequality and the uniform bound $(\ref{unifbound})$ of $\fe$ in $L^2(\Omega;L^2(0,T;L^2_{x,v}))$, that we have
$$
\begin{aligned}\frac{1}{\eps^2}\E\int_0^t\int_{\T^N\times V}\varphi_{\delta}'(\re_s)D^*_s\, \dd s\ =\ \frac{1}{\eps^2}\E\int_0^t\int_{\T^N}J_s\, \dd x\dd s\ &\leq\ \frac{C\delta}{\eps^2}\left(1+\E\int_0^t\|\fe_s\|_{L^1_{x,v}}\dd s\right)\\
&\leq\ \frac{C\delta}{\eps^2}.
\end{aligned}
$$
Let us now study the last term of the It\^o formula. We point out that $\varphi_{\delta}''$ is zero on $[0,\delta]^c$ and that $|\varphi_{\delta}''|\leq 1/\delta$ on $[0,\delta]$. Thus, with $(\ref{H})$, we may write
$$
\E\int_0^t\int_{\T^N\times V}\varphi_{\delta}''(\re_s)G|H_s|^2\,\dd s\ \leq\ \frac{1}{\delta} \E\int_0^t\int_{\T^N\times V}G|H_s|^2\mathbf{1}_{|\re_s|\leq \delta}\,\dd s\ \leq\ \frac{C}{\delta}(\eps^4+\delta^2).
$$
Summing up all the previous bounds now yields
$$\E\|(\re_t)^+\|_{L^1_{x,v}} \leq\  \E\|(\re_{\text{in}})^+\|_{L^1_{x,v}}+\delta + C\eps + \E\int_0^t\|\re_s\|_{L^1_{x,v}}\dd s+\frac{C\delta}{\eps^2}+\frac{C}{\delta}(\eps^4+\delta^2).$$
Now observe that $(\re)^-=(-\re)^+=(\ve-\fe)^+$ to obtain similarly (making use of the bound $(\ref{accr2})$ instead of $(\ref{accr1})$ when applying Lemma \ref{accr})
$$\E\|(\re_t)^-\|_{L^1_{x,v}} \leq\  \E\|(\re_{\text{in}})^-\|_{L^1_{x,v}}+\delta + C\eps + \E\int_0^t\|\re_s\|_{L^1_{x,v}}\dd s+\frac{C\delta}{\eps^2}+\frac{C}{\delta}(\eps^4+\delta^2).$$
Summing the two previous bounds and applying the Gronwall's lemma, we get
$$\E\|\re_t\|_{L^1_{x,v}} \leq  C\left(\E\|\re_{\text{in}}\|_{L^1_{x,v}}+\delta + \eps + \frac{\delta}{\eps^2}+\frac{\eps^4}{\delta}+\delta\right).$$
Since this bound is valid for all $\delta>0$, we choose $\delta=\eps^3$ to discover
$$\E\|\re_t\|_{L^1_{x,v}} \leq  C\left(\E\|\re_{\text{in}}\|_{L^1_{x,v}}+ \eps \right).$$
We point out that $\re_{\text{in}}=-\eps f^1(0)-\eps^2 f^2(0)$, so that 
$$\E\|\re_t\|_{L^1_{x,v}} \leq  C\eps.$$
Finally, thanks to $(\ref{f1infty})$, $(\ref{f2infty})$ and $(\ref{f3infty})$ with $p=1$, we have 
$$\sup\limits_{t\in[0,T]}\E\|\eps f_1+\eps^2 f_2+\eps^3\feth\|_{L^1_{x,v}}\leq C\eps$$
so that we obtain the estimate
$$\sup\limits_{t\in[0,T]}\E\|\fe_t-\rho_t\|_{L^1_{x,v}}\leq C\eps,$$
which concludes the proof of Theorem \ref{mainresult}.

\section*{Appendix A}

\noindent \textbf{Proof of Lemma \ref{accr}.}
\begin{proof}
Let us prove the first estimate; the second one is proved similarly. We are interested in the term 
$$J:=\int_{V}\varphi_{\delta}'(f-g)\left[\sigma(\li{f})L(f)-\sigma(\li{g})L(g)\right]\dd v.$$ 
Here, we observe that 
$$
\begin{array}{ll}
0 & \displaystyle =\ \varphi_{\delta}'(\li{f}-\li{g})\left[\sigma(\li{g})(\li{g}-\li{g})-\sigma(\li{f})(\li{f}-\li{f})\right] \\ [1em]
&\displaystyle = \int_V \varphi_{\delta}'(\li{f}-\li{g})\left[\sigma(\li{g})(\li{g}-g)-\sigma(\li{f})(\li{f}-f)\right] \dd v.
\end{array}
$$
As a consequence, we can write
$$
\begin{aligned}
J &=\int_{V}\varphi_{\delta}'(f-g)\left[\sigma(\li{f})\li{f}-\sigma(\li{f})f-\sigma(\li{g})\li{g}+\sigma(\li{g})g\right]\dd v \\
& +\int_V \varphi_{\delta}'(\li{f}-\li{g})\left[\sigma(\li{g})(\li{g}-g)-\sigma(\li{f})(\li{f}-f)\right] \dd v \\
& = \int_V\left[\sigma(\li{f})\li{f}-\sigma(\li{g})\li{g}\right]\left[\varphi_{\delta}'(f-g)-\varphi_{\delta}'(\li{f}-\li{g})\right]\dd v \\
& + \int_V\left[\sigma(\li{f})f-\sigma(\li{g})g\right]\left[\varphi_{\delta}'(\li{f}-\li{g})-\varphi_{\delta}'(f-g)\right]\dd v \\
& =:J_1 + J_2.
\end{aligned}
$$
We will now bound $J_1$ and $J_2$ separately. Let us begin with the case of $J_1$. We decompose $J_1$ as:
$$
\begin{aligned}
J_1 &=\int_V\left[\sigma(\li{f})\li{f}-\sigma(\li{g})\li{g}\right]\left[\varphi_{\delta}'(f-g)-\varphi_{\delta}'(\li{f}-\li{g})\right]\textbf{1}_{f-g\leq 0}\dd v \\
& + \int_V\left[\sigma(\li{f})\li{f}-\sigma(\li{g})\li{g}\right]\left[\varphi_{\delta}'(f-g)-\varphi_{\delta}'(\li{f}-\li{g})\right]\textbf{1}_{\li{f}-\li{g}\leq 0}\dd v \\
& + \int_V\left[\sigma(\li{f})\li{f}-\sigma(\li{g})\li{g}\right]\left[\varphi_{\delta}'(f-g)-\varphi_{\delta}'(\li{f}-\li{g})\right]\textbf{1}_{f-g\in[0,\delta],\li{f}-\li{g}\in[0,\delta]}\dd v \\
& + \int_V\left[\sigma(\li{f})\li{f}-\sigma(\li{g})\li{g}\right]\left[\varphi_{\delta}'(f-g)-\varphi_{\delta}'(\li{f}-\li{g})\right]\textbf{1}_{f-g\in[0,\delta],\li{f}-\li{g}\geq\delta}\dd v \\
& + \int_V\left[\sigma(\li{f})\li{f}-\sigma(\li{g})\li{g}\right]\left[\varphi_{\delta}'(f-g)-\varphi_{\delta}'(\li{f}-\li{g})\right]\textbf{1}_{f-g\geq\delta,\li{f}-\li{g}\in[0,\delta]}\dd v \\
& + \int_V\left[\sigma(\li{f})\li{f}-\sigma(\li{g})\li{g}\right]\left[\varphi_{\delta}'(f-g)-\varphi_{\delta}'(\li{f}-\li{g})\right]\textbf{1}_{f-g\geq\delta,\li{f}-\li{g}\geq\delta}\dd v \\
& =: J_1^{(1)}+J_1^{(2)}+J_1^{(3)}+J_1^{(4)}+J_1^{(5)}+J_1^{(6)}.
\end{aligned}
$$
\indent\textit{Study of $J_1^{(1)}$:} Note that when $f-g\leq 0$, we have $\varphi_{\delta}'(f-g)=0$. If $\li{f}\leq\li{g}$, we also have $\varphi_{\delta}'(\li{f}-\li{g})=0$, and if $\li{f}\geq\li{g}$, we have $\sigma(\li{f})\li{f}-\sigma(\li{g})\li{g}\geq 0$ thanks to the monotonicity of $x\mapsto\sigma(x)x$ (see (H3)) and $\varphi_{\delta}'(\li{f}-\li{g})\in[0,1]$. As a result, we conclude
$$J_1^{(1)}\leq 0.$$
\indent\textit{Study of $J_1^{(2)}$:} Note that when $\li{f}-\li{g}\leq 0$, we have $\varphi_{\delta}'(\li{f}-\li{g})=0$, $\sigma(\li{f})\li{f}-\sigma(\li{g})\li{g}\leq 0$ thanks to the monotonicity of $x\mapsto\sigma(x)x$ and $\varphi_{\delta}'(f-g)\in[0,1]$ so that we obtain
$$J_1^{(2)}\leq 0.$$
\indent\textit{Study of $J_1^{(3)}$:} First, we write
$$J_1^{(3)}=\int_V\left[(\sigma(\li{f})-\sigma(\li{g}))\li{f}+\sigma(\li{g})(\li{f}-\li{g})\right]\left[\varphi_{\delta}'(f-g)-\varphi_{\delta}'(\li{f}-\li{g})\right]\textbf{1}_{f-g\in[0,\delta],\li{f}-\li{g}\in[0,\delta]}\dd v.$$
Since $\varphi_{\delta}'(f-g)-\varphi_{\delta}'(\li{f}-\li{g})\in[-1,1]$, we obtain with (H1) and the Lipschitz continuity of $\sigma$ (see (H2)) that 
$$
\begin{aligned}
J_1^{(3)} &\leq \int_V\left(|\li{f}|\|\sigma\|_{\text{Lip}}\delta+\sigma^*\delta\right)\textbf{1}_{f-g\in[0,\delta],\li{f}-\li{g}\in[0,\delta]}\dd v\\
&\leq C(1+|\li{f}|)\delta.
\end{aligned}
$$
\indent\textit{Study of $J_1^{(4)}$:} Note that when $\li{f}-\li{g}\geq \delta$ we have $\varphi_{\delta}'(\li{f}-\li{g})=1$ and $\sigma(\li{f})\li{f}-\sigma(\li{g})\li{g}\geq 0$ thanks to the monotonicity of $x\mapsto\sigma(x)x$. Since $\varphi_{\delta}'(f-g)\in[0,1]$, we thus get
$$J_1^{(4)}\leq 0.$$
\indent\textit{Study of $J_1^{(5)}$:} Exactly as in the case of $J_1^{(3)}$, we get
$$
\begin{aligned}
J_1^{(5)} &\leq \int_V\left(|\li{f}|\|\sigma\|_{\text{Lip}}\delta+\sigma^*\delta\right)\textbf{1}_{f-g\geq\delta,\li{f}-\li{g}\in[0,\delta]}\dd v\\
&\leq C(1+|\li{f}|)\delta.
\end{aligned}
$$
\indent\textit{Study of $J_1^{(6)}$:} When $f-g\geq\delta$ and $\li{f}-\li{g}\geq \delta$ we have $\varphi_{\delta}'(f-g)=\varphi_{\delta}'(\li{f}-\li{g})=1$ so that 
$$J_1^{(6)}=0.$$
Now, let us study the case of $J_2$. Similarly, we decompose $J_2$ as:
$$
\begin{aligned}
J_2 &=\int_V\left[\sigma(\li{f})f-\sigma(\li{g})g\right]\left[\varphi_{\delta}'(\li{f}-\li{g})-\varphi_{\delta}'(f-g)\right]\textbf{1}_{f-g\leq 0}\dd v \\
& + \int_V\left[\sigma(\li{f})f-\sigma(\li{g})g\right]\left[\varphi_{\delta}'(\li{f}-\li{g})-\varphi_{\delta}'(f-g)\right]\textbf{1}_{\li{f}-\li{g}\leq 0}\dd v \\
& + \int_V\left[\sigma(\li{f})f-\sigma(\li{g})g\right]\left[\varphi_{\delta}'(\li{f}-\li{g})-\varphi_{\delta}'(f-g)\right]\textbf{1}_{f-g\in[0,\delta],\li{f}-\li{g}\in[0,\delta]}\dd v \\
& + \int_V\left[\sigma(\li{f})f-\sigma(\li{g})g\right]\left[\varphi_{\delta}'(\li{f}-\li{g})-\varphi_{\delta}'(f-g)\right]\textbf{1}_{f-g\in[0,\delta],\li{f}-\li{g}\geq\delta}\dd v \\
& + \int_V\left[\sigma(\li{f})f-\sigma(\li{g})g\right]\left[\varphi_{\delta}'(\li{f}-\li{g})-\varphi_{\delta}'(f-g)\right]\textbf{1}_{f-g\geq\delta,\li{f}-\li{g}\in[0,\delta]}\dd v \\
& + \int_V\left[\sigma(\li{f})f-\sigma(\li{g})g\right]\left[\varphi_{\delta}'(\li{f}-\li{g})-\varphi_{\delta}'(f-g)\right]\textbf{1}_{f-g\geq\delta,\li{f}-\li{g}\geq\delta}\dd v \\
& =: J_2^{(1)}+J_2^{(2)}+J_2^{(3)}+J_2^{(4)}+J_2^{(5)}+J_2^{(6)}.
\end{aligned}
$$
\indent\textit{Study of $J_2^{(1)}$:} When $f-g\leq 0$, we have $\varphi_{\delta}'(f-g)=0$. If $\li{f}\leq\li{g}$, we also have $\varphi_{\delta}'(\li{f}-\li{g})=0$; and if $\li{f}\geq\li{g}$, we have $\sigma(\li{f})f-\sigma(\li{g})g\leq 0$ thanks to the monotonicity of $\sigma$ (see (H3)) and the positivity of $f$. Since  $\varphi_{\delta}'(\li{f}-\li{g})\in[0,1]$, we conclude
$$J_1^{(1)}\leq 0.$$
\indent\textit{Study of $J_2^{(2)}$:} When $\li{f}-\li{g}\leq 0$, we have $\varphi_{\delta}'(\li{f}-\li{g})=0$ and $\sigma(\li{f})-\sigma(\li{g})\geq 0$ thanks to the monotonicity of $\sigma$. If $f\leq g$, we also have $\varphi_{\delta}'(f-g)=0$. If $f\geq g\geq 0$, we have $\sigma(\li{f})f-\sigma(\li{g})g\geq 0$. If $f\geq 0\geq g$, we still have $\sigma(\li{f})f-\sigma(\li{g})g\geq 0$ since $\sigma\geq 0$. Note that the case $0\geq f\geq g$ is impossible by positivity of $f$. Finally, since  $\varphi_{\delta}'(f-g)\in[0,1]$, we conclude
$$J_2^{(2)}\leq 0.$$
\indent\textit{Study of $J_2^{(3)}$:} First, we write
$$J_2^{(3)}=\int_V\left[(\sigma(\li{f})-\sigma(\li{g}))f+\sigma(\li{g})(f-g)\right]\left[\varphi_{\delta}'(\li{f}-\li{g})-\varphi_{\delta}'(f-g)\right]\textbf{1}_{f-g\in[0,\delta],\li{f}-\li{g}\in[0,\delta]}\dd v.$$
Since $\varphi_{\delta}'(f-g)-\varphi_{\delta}'(\li{f}-\li{g})\in[-1,1]$, we obtain with (H1) and the Lipschitz continuity of $\sigma$ that 
$$
\begin{aligned}
J_2^{(3)} &\leq \int_V\left(|f|\|\sigma\|_{\text{Lip}}\delta+\sigma^*\delta\right)\textbf{1}_{f-g\in[0,\delta],\li{f}-\li{g}\in[0,\delta]}\dd v\\
&\leq C(1+\li{|f|})\delta.
\end{aligned}
$$
\indent\textit{Study of $J_2^{(4)}$:} 
We write
$$J_2^{(4)}=\int_V\left[(\sigma(\li{f})-\sigma(\li{g}))f+\sigma(\li{g})(f-g)\right]\left[\varphi_{\delta}'(\li{f}-\li{g})-\varphi_{\delta}'(f-g)\right]\textbf{1}_{f-g\in[0,\delta],\li{f}-\li{g}\geq \delta}\dd v.$$
Note that when $\li{f}-\li{g}\geq \delta$ we have $\varphi_{\delta}'(\li{f}-\li{g})-\varphi_{\delta}'(f-g)=1-\varphi_{\delta}'(f-g)\in[0,1]$ and $\sigma(\li{f})-\sigma(\li{g})\leq 0$ thanks to the monotonicity of $\sigma$. With the positivity of $f$, we thus get
$$J_2^{(4)}\leq \sigma^*\delta.$$
\indent\textit{Study of $J_2^{(5)}$:} 
We have
$$J_2^{(5)}=\int_V\left[(\sigma(\li{f})-\sigma(\li{g}))f+\sigma(\li{g})(f-g)\right]\left[\varphi_{\delta}'(\li{f}-\li{g})-\varphi_{\delta}'(f-g)\right]\textbf{1}_{f-g\geq\delta,\li{f}-\li{g}\in[0,\delta]}\dd v.$$
Note that when $f-g\geq \delta$ we have $\varphi_{\delta}'(\li{f}-\li{g})-\varphi_{\delta}'(f-g)=\varphi_{\delta}'(\li{f}-\li{g})-1\in[-1,0]$. We thus get
$$J_2^{(5)}\leq \|\sigma\|_{\text{Lip}}\li{|f|}\delta.$$
\indent\textit{Study of $J_2^{(6)}$:} When $f-g\geq\delta$ and $\li{f}-\li{g}\geq \delta$ we have $\varphi_{\delta}'(f-g)=\varphi_{\delta}'(\li{f}-\li{g})=1$ so that 
$$J_2^{(6)}=0.$$
To sum up, we get the following bound on $J$
$$J\leq C(1+\|f\|_{L^1_v})\delta,$$
which concludes the proof.
\end{proof}

\section*{Appendix B}

\noindent \textbf{Proof of Lemma \ref{lemmemoyenne}.} We recall the Lemma to be proved.

\begin{lemmanonnum}
Let $\alpha \in (0,1]$. We assume that hypothesis \eqref{nondegenlemmemoy} is satisfied. Let $f$ be bounded in $L^2(\Omega;L^2(0,T;L^2_{x,v}))$ such that
\begin{equation}\label{eqmoyenne}
\dd f + a(v)\cdot\nabla_x f \dd t = h \dd t + g\, Q\dd W_t,
\end{equation}
with $g$ and $h$ bounded in $L^2(\Omega;L^2(0,T;L^2_{x,v}))$. Then the quantity $\rho = \li{f}$ verifies
$$\E\int_0^T\|\rho_s\|^2_{H^{\alpha/2}_x}\dd s\leq C.$$
\end{lemmanonnum}

\begin{proof}
We adapt in our stochastic context the proof of \cite[Theorem 2.3]{bouchut}. We recall that $Q\dd W_t = \sum_{k\geq 0} Qe_k \dd\beta_k(t)$ but, in order to simplify the notations, we assume in the proof that the noise is one-dimensional, namely of the form $Qe_l \dd\beta_l(t)$, $l\geq 0$, the generalization to an infinite dimensional noise being straightforward. We set $\theta_l = Qe_l$. Let $k\in\Z^N\mapsto \ti{f}(k)$ denote the Fourier transform of $f$ with respect to the space variable $x\in\T^N$. We take the spatial Fourier transform in Equation \eqref{eqmoyenne} and we add artificially on both sides of the equation a term $\lambda \ti{f}$ for some constant $\lambda> 0$ to be chosen later. We obtain, for $k\in\Z^N$,
$$d\ti{f}(k) - ia(v)\cdot k \ti{f}(k)\dd t + \lambda \ti{f}(k) = \ti{h}\dd t + \ti{g\theta_l} \dd \beta_l(t) + \lambda \ti{f}(k).$$
Using Duhamel's formula, we have
$$
\begin{aligned}
\ti{f}(t,k,v) &= e^{-(\lambda-ia(v)\cdot k)t}\ti{f}(0,k,v) + \int_0^t e^{-(\lambda-ia(v)\cdot k)(t-s)}[\ti{h}+\lambda\ti{f}](s,k,v)\,\dd s \\
&+ \int_0^te^{-(\lambda-ia(v)\cdot k)(t-s)}\ti{g\theta_l}(s,k,v)\,\dd \beta_l(s).
\end{aligned}
$$ 
Integrating in the velocity variable $v\in V$, we get
$$
\begin{aligned}
\ti{\rho}(t,k) &= e^{-\lambda t} \int_V e^{ia(v)\cdot kt}\ti{f}(0,k,v)\dd v + \int_0^te^{-\lambda(t-s)}\int_V e^{ia(v)\cdot k(t-s)}[\ti{h}+\lambda\ti{f}](s,k,v)\,\dd v\,\dd s \\
&+ \int_0^te^{-\lambda(t-s)}\int_V e^{ia(v)\cdot k(t-s)}\ti{g\theta_l}(s,k,v)\,\dd v\, \dd \beta_l(s) \\
& =  T_d(t,k) + T_s(t,k),
\end{aligned}
$$
where 
$$T_d(t,k) := e^{-\lambda t} \int_V e^{ia(v)\cdot kt}\ti{f}(0,k,v)\dd v + \int_0^te^{-\lambda(t-s)}\int_V e^{ia(v)\cdot k(t-s)}[\ti{h}+\lambda\ti{f}](s,k,v)\,\dd v\,\dd s
$$ and 
$$ T_s(t,k) :=  \int_0^te^{-\lambda(t-s)}\int_V e^{ia(v)\cdot k(t-s)}\ti{g\theta_l}(s,k,v)\,\dd v\, \dd \beta_l(s) $$
denote respectively the deterministic and stochastic part of $\ti{\rho}(t,k)$. Let $k\in\Z^N$, $k\neq 0$. The deterministic term can be handled exactly as in the proof of \cite[Theorem 2.3]{bouchut} and we obtain, up to a real multiplicative constant,
$$\E\int_0^T|T_d|^2(t,k)\,\dd t \leq \frac{1}{\lambda^{1-\alpha} |k|^{\alpha}}\E\int_V |\ti{f}|^2(0,k,v)\, \dd v + \frac{1}{\lambda^{2-\alpha} |k|^{\alpha}}\E\int_0^T\!\!\!\int_V |\ti{h}+\lambda\ti{f}|^2(s,k,v)\, \dd v\,\dd s.$$
So let us now focus on the stochastic term $T_s$. First, using the It\^o isometry, we have
$$
\begin{aligned}
\E|T_s|^2(t,k) &= \E \int_0^t e^{-2\lambda(t-s)}\Big{|}\int_V e^{ia(v)\cdot k(t-s)}\ti{g\theta_l}(s,k,v)\,\dd v\Big{|}^2\, \dd s \\
&= \E\int_0^t e^{-2\lambda s}\Big{|}\int_V e^{ia(v)\cdot k s}\ti{g\theta_l}(t-s,k,v)\,\dd v\Big{|}^2\, \dd s,
\end{aligned}
$$
so that, by the Fubini Theorem and the change of variable $\tau := t-s$, we have
$$
\begin{aligned}
\E\int_0^T |T_s|^2(t,k)\,\dd t &= \E\int_0^T\!\!\!\int_0^{T-\tau} \!\!\!e^{-2\lambda s}\Big{|}\int_V e^{ia(v)\cdot k s}\ti{g\theta_l}(\tau,k,v)\,\dd v\Big{|}^2\, \dd s \,\dd \tau \\
& \leq \E\int_0^T\!\!\int_{\R} e^{-2\lambda s}\Big{|}\int_V e^{ia(v)\cdot k s}\ti{g\theta_l}(\tau,k,v)\,\dd v\Big{|}^2\, \dd s \,\dd \tau \\
& = \frac{1}{|k|}\E\int_0^T\!\!\int_{\R} e^{-\frac{2\lambda s}{|k|}}\Big{|}\int_V e^{ia(v)\cdot \frac{k}{|k|} s}\ti{g\theta_l}(\tau,k,v)\,\dd v\Big{|}^2\, \dd s \,\dd \tau.
\end{aligned}
$$
We use the bound
$$e^{-\frac{2\lambda s}{|k|}} \leq \frac{1}{1+\frac{4\lambda}{|k|^2}s^2}, \quad s\geq 0,$$
and estimate the oscillatory integral thanks to \cite[Lemma 2.4]{bouchut} and \eqref{nondegenlemmemoy}; we therefore get
$$
\E\int_0^T |T_s|^2(t,k)\,\dd t \leq  \frac{C}{\lambda^{1-\alpha} |k|^{\alpha}}\E\int_0^T \!\!\int_{V} |\ti{g\theta_l}|^2(\tau,k,v)\,\dd v\,\dd \tau.
$$
As a result, summing up the previous bounds, we have, up to a real multiplicative constant, 
\begin{multline*}
\E\int_0^T |\ti{\rho}|^2(t,k)\,\dd t \leq \frac{1}{\lambda^{1-\alpha} |k|^{\alpha}}\E\int_0^T \!\!\int_{V} |\ti{g\theta_l}|^2(\tau,k,v)\,\dd v\,\dd \tau + \frac{1}{\lambda^{1-\alpha} |k|^{\alpha}}\E\int_V |\ti{f}|^2(0,k,v)\, \dd v \\
+ \frac{1}{\lambda^{2-\alpha} |k|^{\alpha}}\E\int_0^T\!\!\!\int_V |\ti{h}+\lambda\ti{f}|^2(s,k,v)\, \dd v\,\dd s.
\end{multline*}
We choose $\lambda \equiv 1$, multiply the last equation by $|k|^{\alpha}$ and sum over $k\in\Z^N$ to find
$$
\begin{aligned}
\E\int_0^T\|\rho(t)\|^2_{H^{\alpha/2}_x}\dd t &\leq C\E\left[\|g\theta_l\|^2_{L^2(0,T;L^2_{x,v})} + \|h+f \|^2_{L^2(0,T;L^2_{x,v})}+\|f(0) \|^2_{L^2_{x,v}}\right] \\
&\leq  C\E\left[\|Qe_l\|^2_{L^{\infty}_x}\|g\|^2_{L^2(0,T;L^2_{x,v})} + \|h+f \|^2_{L^2(0,T;L^2_{x,v})}+\|f(0) \|^2_{L^2_{x,v}}\right].
\end{aligned}
$$
This concludes the proof when the noise is finite dimensional. For the infinite dimensional case, we recall that, thanks to \eqref{regularnoise01}, we have $\kappa_{0,\infty}=\sum_{l\geq 0}\|Qe_l\|^2_{L^{\infty}_x}<\infty$.
\end{proof}

\bibliographystyle{plain}
\bibliography{biblio}
\nocite{daprato}
\nocite{rosseland}
\end{document}